\documentclass{article}
\usepackage{
amsmath,
amsfonts,
latexsym, 
amsrefs,
amssymb
}

\usepackage{bm}

\usepackage{tocloft}

\usepackage{cite}

\usepackage{pstricks}

\usepackage{subcaption}

\usepackage{tikz}
\usepackage{tikz,fullpage}
\usetikzlibrary{arrows,%
                petri,%
                topaths, automata}%

\usepackage{graphicx}
\usepackage{enumerate}
\usepackage{mathrsfs}
\usepackage{wrapfig}

\usepackage{color}

\newcommand{\Id}{{\mathrm{Id}}}

\numberwithin{equation}{section}
\newcommand{\bla}{\bm{\lambda}}
\newcommand{\bmu}{\bm{\mu}}

\newcommand{\rd}{{\rm d}}

\newcommand{\bx}{{\bf{x}}}

\newcommand{\bv}{{\bf{v}}}
\newcommand{\bw}{{\bf{w}}}

\newcommand{\be}{\begin{equation}}
\newcommand{\ee}{\end{equation}}

\newcommand{\e}{{\varepsilon}}

\newcommand{\la}{\lambda}

\newcommand{\rU}{{\rm U}}

\usepackage{amsmath} 
\usepackage{amssymb}
\usepackage{amsthm}

\newcommand{\wt}{\widetilde}

\newcommand{\ii}{\mathrm{i}} 

\renewcommand{\epsilon}{\varepsilon}
\renewcommand{\leq}{\leqslant}
\renewcommand{\geq}{\geqslant}

\renewcommand{\le}{\leq}

\newcommand{\E}{\mathbb{E}}

\newcommand{\pB}[1]{\Bigl({#1}\Bigr)}

\DeclareMathOperator{\var}{Var}

\DeclareMathOperator{\re}{Re}
\DeclareMathOperator{\im}{Im}

\DeclareMathOperator{\OO}{O}
\DeclareMathOperator{\oo}{o}

\theoremstyle{plain} 
\newtheorem{theorem}{Theorem}[section]
\newtheorem*{theorem*}{Theorem}
\newtheorem{lemma}[theorem]{Lemma}
\newtheorem*{lemma*}{Lemma}
\newtheorem{corollary}[theorem]{Corollary}
\newtheorem*{corollary*}{Corollary}
\newtheorem{proposition}[theorem]{Proposition}
\newtheorem*{proposition*}{Proposition}

\newtheorem{definition}[theorem]{Definition}
\newtheorem*{definition*}{Definition}

\newtheorem*{example*}{Example}
\newtheorem{remark}[theorem]{Remark}

\newtheorem*{remark*}{Remark}
\newtheorem*{remarks*}{Remarks}

\makeatletter
\renewcommand{\subsection}{\@startsection
{subsection}
{2}
{0mm}
{-\baselineskip}
{0 \baselineskip}
{\normalfont\bf\itshape}} 
\makeatother

\usepackage{dsfont}
\usepackage{stmaryrd}

\newcommand{\nc}{\normalcolor}

\def\@empty{}

\def\author#1{\par
    {\centering{\authorfont#1}\par\vspace*{0.05in}}
}

\def\titlefont{\fontsize{13}{15}\bfseries\boldmath\selectfont\centering{}}
\def\authorfont{\fontsize{13}{15}}
\def\abstractfont{\fontsize{8}{10}}

\let\affiliationfont\rhfont

\def\address#1{\par
    {\centering{\affiliationfont#1\par}}\par\vspace*{11pt}
}

\def\body{
\setcounter{footnote}{0}
\def\thefootnote{\alph{footnote}}
\def\@makefnmark{{$^{\rm \@thefnmark}$}}
}

\def\title#1{
    \thispagestyle{plain}
    \vspace*{-14pt}
    \vskip 79pt
    {\centering{\titlefont #1\par}}%
    \vskip 1em
}

\renewenvironment{abstract}{\par%
    \vspace*{6pt}\noindent 
    \abstractfont
    \noindent\leftskip10pt\rightskip10pt
}{%
  \par}

\newcommand{\f}[1]{\boldsymbol{\mathrm{#1}}} 
\newcommand{\for}{\qquad \text{for} \quad}

\usepackage{tocloft}

\makeatletter
\renewcommand{\section}{\@startsection
{section}
{1}
{0mm}
{-2\baselineskip}
{1\baselineskip}
{\normalfont\large\scshape\centering}} 
\makeatother

\begin{document}

~\vspace{-1.4cm}

\title{Extreme gaps between eigenvalues of Wigner matrices}

\vspace{1cm}
\noindent

\begin{minipage}[b]{0.3\textwidth}
\hspace{3cm}
 
 \end{minipage}
 \begin{minipage}[b]{0.3\textwidth}
 \author{P. Bourgade}

\address{Courant Institute\\
   E-mail: bourgade@cims.nyu.edu}
 \end{minipage}

\begin{minipage}[b]{0.3\textwidth}

 \end{minipage}

\begin{abstract}
This paper proves universality of the distribution of the smallest and largest gaps between eigenvalues of  generalized Wigner matrices, under some smoothness assumption for the density of the entries.

The proof relies on the Erd{\H o}s-Schlein-Yau dynamic approach. We exhibit a new observable that satisfies a stochastic advection equation and reduces local relaxation of the Dyson Brownian motion to a maximum principle.
This observable also provides a simple and unified proof of gap universality in the bulk and the edge, which is quantitative.
To illustrate this, we give the first explicit rate of convergence to the Tracy-Widom distribution for generalized Wigner matrices.
\end{abstract}

\vspace{-0.3cm}

\tableofcontents

\vspace{0.2cm}

\section{Introduction}\

\vspace{-0.5cm}

\subsection{Extreme statistics in random matrix theory.}\ 
The study of extreme spacings in random spectra was initially limited to integrable models. Vinson
\cite{Vin2001} showed that the smallest gap between eigenvalues of the $N\times N$ Circular Unitary Ensemble, multiplied by $N^{4/3}$, has limiting density $3x^2e^{-x^3}$, as the size $N$ increases.
In his thesis, similar results for the smallest gap between eigenvalues of a generalization of the Gaussian Unitary Ensemble were obtained.
With a different method Soshnikov \cite{Sos2005} computed the
distribution of the smallest gap for general translation invariant  determinantal point processes in large boxes:
properly rescaled the smallest gap converges, with the same limiting distribution function $e^{-x^3}$.
Vinson also gave heuristics suggesting that the largest gap between eigenvalues in the bulk
should be of order $\sqrt{\log N}/N$, with Poissonian fluctuations around this limit, a problem popularized by Diaconis
\cite{Dia2003}.
Ben Arous and the author addressed this problem concerning the first order asymptotics for the maximum gap, and described the limiting process of small gaps, for CUE and GUE   \cite{BenBou2013}. These results were extended by Figalli and Guionnet to some invariant multimatrix Hermitian ensembles \cite{FigGui2016}. The convergence in distribution of the largest gap was recently solved by Feng and Wei, also for CUE and GUE \cite{FengWei2018II}.
Feng and Wei also investigated the smallest gaps beyond the determinantal case, characterizing their asymptotics for the circular $\beta$ ensembles \cite{FengWei2018I}. For the Gaussian orthogonal ensemble,  together with Tian they proved that the smallest gap rescaled by $N^{3/2}$ converges with limiting density function $2xe^{-x^2}$ \cite{FengWei2018III}.

The intuition for all results above are (i) the Poissonian ansatz, namely the eigenvalues' gaps are asymptotically independent, (ii) weak convergence of the spacings holds with good convergence rate, so that the finite $N$ gap density asymptotics at $0^+$ and $\infty$ are close to the limiting Gaudin density asymptotics.

This paper shows that the above limit theorems and heuristic picture hold beyond invariant ensembles. In particular, the gap universality for Wigner matrices by Erd{\H o}s and Yau \cite{ErdYau2015}
extends to submicroscopic scales. We informally state this optimal separation of eigenvalues as follows
(see Theorem \ref{thm:smallprocess} for details, in particular the smoothness assumption).

\begin{theorem*} Let  $\lambda_1<\dots<\lambda_N$ be the eigenvalues of a symmetric Wigner matrix with entries satisfying some weak smoothness assumption. Then for any small $\kappa>0$ there exists $c>0$ such that for any $x>0$
$$
\lim_{N\to\infty}\mathbb{P}\left(cN^{\frac{3}{2}}\min_{\kappa N\leq i\leq (1-\kappa)N}(\lambda_{i+1}-\lambda_i)>x\right)= e^{-x^2}.
$$
\end{theorem*}
\noindent The same result holds for the Hermitian class, with rescaling $N^{4/3}$  and limit
$e^{-x^3}$.  Our work also applies to universality of the largest gaps (see Theorem \ref{thm:largest}), under similar assumptions.\\

\noindent For the proof, we develop a new approach to the analysis of the Dyson Brownian motion  (see Subsection \ref{sub:sketch}). Relaxation of eigenvalues simply follows from a the new observable (\ref{eqn:observable}) which satisfies a stochastic advection equation.\\

\noindent Does the above theorem require our slight smoothness hypothesis (\ref{eqn:smooth}) on the matrix entries? For the largest gaps, which are essentially on the microscopic scale $1/N$, this assumption is  unnecessary as shown  by Landon, Lopatto and Marcinek in the simultaneous work \cite{LanLopMar2018}. The scale of the smallest gaps is  harder to access: the current best lower bound on separation of eigenvalues for Wigner matrices 
with atomic distribution is $N^{-2+\oo(1)}$, by Nguyen, Tao and Vu \cite{NguTaoVu2017} (see also \cite{LopLuh2019} for the case of sparse matrices).\\

Motivations for the extreme eigenvalues' gaps statistics include relaxation time for diagonalization algorithms \cite{DeiTro2017, BenBou2013}, conjectures in analytic number theory (e.g. the extreme gaps between zeros of the Riemann zeta function \cite{BenBou2013,BuiMil2018}), conjectures in algorithmic number theory 
(the Poisson ansatz for large gaps suggests the complexity of an algorithm to detect square free numbers \cite{BooHiaKea2015}), and quantum chaos in the complementary Poissonian regime \cite{BloBouRadRud2017}.

Another motivation for extreme value statistics  in random matrix theory  emerged  after the work of 
Fyodorov, Hiary and  Keating \cite{FyoHiaKea2012}: the maximum of the characteristic polynomial of random matrices 
predicts the scale and fluctuations of the maximum of the Riemann zeta function on typical intervals of the critical line.
Recent progress about their conjecture verified the size of the maximum of the characteristic polynomial, for integrable random matrices \cite{ArgBelBou2017,PaqZei2018,ChhMadNaj2018,LamPaq2018}.
We expect that the observable (\ref{eqn:observable}) will also help understanding universality for such extreme statistics.
Indeed it was an important tool in the recent proof of fluctuations of determinants of Wigner matrices \cite{BouMod2018}.

\subsection{Results on extreme gaps.}\ 
We will use the notation $a_N\sim b_N$ if there exists $C>0$ such that $C^{-1} b_N\leq a_N\leq C b_N$ for all $N$. 
In this work, we consider the following class of random matrices. 

\begin{definition}\label{def:wig}
A generalized Wigner matrix $H=H(N)$ is a  Hermitian or symmetric $N\times N$ matrix whose upper-triangular elements $H_{ij}=\overline{H_{ji}}$, $i\leq j$, are independent random variables with mean zero and variances $\sigma_{ij}^2=\E(|H_{ij}|^2)$  that satisfy  the following two conditions:
\begin{enumerate}[(i)]
\vspace{-0.1cm}
\item Normalization: for any $j\in\llbracket 1,N\rrbracket$, $\sum_{i=1}^N\sigma_{ij}^2=1$.
\vspace{-0.1cm}
\item Non-degeneracy:  $ \sigma_{ij}^2\sim N^{-1}$ for all $i,j\in\llbracket 1,N\rrbracket$. 
\end{enumerate}
In the Hermitian case, we
assume $\var\re(H_{ij})\sim \var\im(H_{ij})$  and independence of  $\re(H_{ij}), \im(H_{ij})$\footnote{ 
Other assumptions would work, such as the law of $H_{ij}$ being isotropic. We consider the independent case for simplicity.}.
\end{definition}

We also suppose for convenience (this could be replaced by a finite large moment assumption) that the matrix entries satisfy a tail  estimate:  there exists  $c>0$  such that for any $i,j,N$ and $x>0$ we have
\begin{equation}\label{eqn:tail}
\mathbb{P}\left(|\sqrt{N}H_{ij}|>x\right)\leq c^{-1}e^{-x^{c}}.
\end{equation}
We denote the limiting spectral density of Wigner matrices
\begin{equation*}
\rho(x)=\frac{1}{2\pi}\sqrt{(4-x^2)_+}.
\end{equation*}

In some of the following results, we additionally assume non-atomicity for the matrix entries. A sequence $(H_N)_N$ of random matrices is said to be smooth on scale $\sigma=\sigma(N)$ if
$\sqrt{N}H_{ij}$ has density $e^{-V}$, where $V=V_{N,i,j}$ satisfies the following condition
uniformly in $N,i,j$.
For any $k\geq 0$ there exists $C>0$ such that
\begin{equation}\label{eqn:smooth}
|V^{(k)}(x)|\leq C\, \sigma^{-k}(1+|x|)^{C},\ x\in\mathbb{R}.
\end{equation}
Finally, we always order the eigenvalues $\lambda_1\leq \dots\leq \lambda_N$ and define the process of small gaps and their position
$$
\chi^{(N)}=\sum_{i=1}^N\delta_{(N^{\frac{\beta+2}{\beta+1}}(\lambda_{i+1}-\lambda_i),\lambda_i)}\mathds{1}_{|\lambda_i|<2-\kappa},
$$
where $\beta=1$ for the generalized Wigner symmetric ensemble and $\beta=2$ for the Hermitian one. The following theorem generalizes (and relies on comparison with) the GUE and GOE cases \cite{BenBou2013,FengWei2018III}\footnote{
Our normalization choice from Definition \ref{def:wig} yields a limiting eigenvalue distribution supported on $[-2,2]$, while \cite{FengWei2018III} gives a support $[-\sqrt{2N},\sqrt{2N}]$. The $\beta=1$ cases in Theorem \ref{thm:smallprocess} and Corollary \ref{cor:smallest} agree with the results from \cite{FengWei2018III} up to this rescaling.}.

\begin{theorem}[Small gaps process]\label{thm:smallprocess}
Let $(H_N)$ be  generalized Wigner matrices  satisfying (\ref{eqn:tail}). Let $\kappa>0$.
\begin{enumerate}[(i)]
\item{\it Symmetric class.}\ 
Assume $(H_N)$ is smooth on scale $\sigma=N^{-1/4+\e}$ for some fixed $\e>0$, in the sense of (\ref{eqn:smooth}).
The point process  $\chi^{(N)}$ converges as $N\to\infty$ to a Poisson point process $\chi$ with intensity given, 
for any measurable sets $A\subset \mathbb{R}_+$ and $I\subset(-2+\kappa,2-\kappa)$, by
$$
\E\chi(A\times I)={\frac{1}{48\pi}}\left(\int_A u\rd u\right)\left(\int_I(4-x^2)^{\frac{3}{2}}\rd x\right).
$$
\item{\it Hermitian class.}
Assume $(H_N)$ is smooth on scale $\sigma=N^{-1/3+\e}$ for some fixed $\e>0$.
The point process  $\chi^{(N)}$ converges  to a Poisson point process $\chi$ with intensity
$$
\E\chi(A\times I)=\frac{1}{48\pi^2}\left(\int_A u^2\rd u\right)\left(\int_I(4-x^2)^2\rd x\right).
$$
\end{enumerate}
\end{theorem}

As a corollary, the distribution of the smallest gaps  in the bulk of the spectrum is explicit.
For the statement, let $t_1=\min\{\lambda_{i+1}-\lambda_i:\lambda_i\in I\}$ be the smallest gap in some interval $I$,  $t_2=\min\{\lambda_{i+1}-\lambda_i:\lambda_i\in I, \lambda_{i+1}-\lambda_i>t_1\}$ the second smallest gap, and analogously for any $t_k$. 
To quantify the speed of convergence below, we consider the Wasserstein distance on $\mathbb{R}$ ($\Gamma$  is the set of all couplings  of $X$ and $Y$), 
\begin{equation}
\label{eqn:Wass}
{\rm d}_{\rm W}(X,Y)=\int  |\mathbb{P}(X\leq x)-\mathbb{P}(Y\leq x)|\rd x
=
\sup_{\|h\|_{\rm Lip}\leq 1}|\E(h(X)-h(Y))|
=\inf_{\gamma\in\Gamma}\int|x-y|\rd\gamma(x,y).
\end{equation}

\begin{corollary}[Smallest gaps]\label{cor:smallest} Assume $(H_N)_N$ is as in Theorem \ref{thm:smallprocess}, $k$ is fixed, $\kappa>0$ and consider a non-empty interval $I\subset(-2+\kappa,2-\kappa)$. 
\begin{enumerate}[(i)]
\item{\it Symmetric class.}\ 
Let $\tau_k=\left(\int_I(4-x^2)^{3/2}\rd x/(96\pi)\right)^{1/2}N^{3/2} t_k$. Then for any interval $J\subset \mathbb{R}_+$, we have
$$
\lim_{N\to\infty}\mathbb{P}(\tau_k\in J)=\int_J\frac{2}{(k-1)!}x^{2k-1}e^{-x^2}\rd x.
$$
The rate of convergence satisfies
$
{\rm d}_{{\rm W}}(\tau_k(H),\tau_k({\rm GOE}))\leq N^c/(N^{1/2}\sigma^2)
$
for any $c>0$.
\item{\it Hermitian class.}
Let $\tau_k=\left(\int_I(4-x^2)^2\rd x/(144\pi^2)\right)^{1/3}N^{4/3} t_k$. Then for any interval $J\subset \mathbb{R}_+$, we have
$$
\lim_{N\to\infty}\mathbb{P}(\tau_k\in J)=\int_J\frac{3}{(k-1)!}x^{3k-1}e^{-x^3}\rd x.
$$
The rate of convergence satisfies
$
{\rm d}_{{\rm W}}(\tau_k(H),\tau_k({\rm GUE}))\leq  N^c/(N^{2/3}\sigma^2)
$
for any $c>0$.
\end{enumerate}
\end{corollary}

There are at least two ways to understand the above scaling  of the smallest spacings, denoted $\ell=N^{-3/2}$ for $\beta=1$,
$\ell=N^{-4/3}$ for $\beta=2$. 
  First, in the Gaussian integrable case, the eigenvalues interaction  $\prod_{i<j}|\la_i-\la_j|^\beta$ suggests  $\mathbb{P}(N(\la_{i+1}-\la_i)<x)\sim x^{\beta+1}$ uniformly in small $x$ and $i$, so that decorrelation of spacings would give
$N (N\ell)^{\beta+1}\sim 1$. 
Second, the resolvent method gives Wegner estimates for Wigner matrices with smooth entries \cite{ErdSchYau2010}. For example,  \cite[Corollary B.2]{BouErdYauYin2016} shows $\mathbb{P}(N(\lambda_{i+1}-\lambda_i)<x)\leq C N^{\e}x^2$ for GOE. A union bound on these level repulsion estimates provides a lower estimate on the smallest gaps, which matches our order.\\

For the largest gaps, Gumbel fluctuations are expected, with heuristics also relying on decoupling, and the asymptotics $e^{-c x^2}$ for the upper tail distribution of  $N(\lambda_{i+1}-\lambda_i)$. However, for the integrable Gaussian ensembles these facts have been  established only for $\beta=2$, thanks to the determinantal structure. We therefore only state the following theorem for the Hermitian class. It proceeds by comparison with the GUE case from \cite{FengWei2018II}. 
May the analogue for GOE be known, the universality would follow. 

As in  \cite{FengWei2018II}, for any interval $I$ we denote $S(I)=\inf_{I}\sqrt{4-x^2}$. 
Let $t_1^*=\max\{\lambda_{i+1}-\lambda_i:\lambda_i\in I\}$ be the largest gap,  $t_2^*=\max\{\lambda_{i+1}-\lambda_i:\lambda_i\in I, \lambda_{i+1}-\lambda_i<t^*_1\}$ the second smallest gap, and analogously for any $t^*_k$.
We rescale the $k$th largest gaps as
$$
\tau_k^*=(2\log N)^{1/2} (N S(I) t^*_{k}-(32\log N)^{1/2})/4+(5/8)\log (2\log N).
$$

\begin{theorem}[Largest gaps in the bulk, Hermitian case]\label{thm:largest}
Let $(H_N)$ be generalized Wigner matrices from the Hermitian class, satisfying (\ref{eqn:tail}) and smooth on scale $\sigma>N^{-\frac{1}{2}+\e}$ for some fixed $\e>0$, in the sense of (\ref{eqn:smooth}). 
Let  $I=[a,b]\subset(-2,2)$. Assume $|a|\leq|b|$ without loss of generality, and define 
$c=(1/12)\log 2+3\zeta'(-1)+(3/2)\log(4-b^2)-\log (4|b|)+(\log 2)\mathds{1}_{a=-b}$. 
For any fixed $k$  and interval $J$, we have
$$
\lim_{N\to\infty}\mathbb{P}(\tau^*_k\in J)=\int_J\frac{e^{k(c-x)}}{(k-1)!}e^{-e^{c-x}}\rd x.
$$
Moreover, the rate of convergence is bounded by
$
{\rm d}_{{\rm W}}(\tau^*_k(H),\tau^*_k({\rm GUE}))\leq N^c/(N\sigma^2)
$
for any $c>0$.
\end{theorem}

\subsection{Results on quantitative universality and eigenvalues' fluctuations.}\ The previous theorems rely on a quantitative relaxation of the Dyson Brownian motion, explained in subsection \ref{sub:sketch}. As a different application, universality holds with explicit rate of convergence, answering a
recurring question, see e.g. \cite{Open2010}. 

We illustrate this at the edge only to keep technicalities minimal, although the method would also give some explicit rate for gaps in the bulk. A non-quantitative convergence to the Tracy Widom distribution was first proved in \cite{Sos1999,TaoVu2010,ErdYauYin2012} for Wigner and in \cite{BouErdYau2014} for generalized Wigner matrices. We consider the Kolmogorov distance
$$
{\rm d}_{\rm K}(X,Y)=\sup_x |\mathbb{P}(X\leq x)-\mathbb{P}(Y\leq x)|.
$$

\begin{theorem}\label{thm:rate}
Let $(H_N)$ be  generalized Wigner matrices from the symmetric ($\beta=1$) or Hermitian ($\beta=2$) class satisfying (\ref{eqn:tail}). Denoting ${\rm TW}_{\beta}$ the corresponding limiting Tracy-Widom distribution, for any $c>0$, for large enough $N$ we have
$$
{\rm d}_{{\rm K}}(N^{2/3}(\lambda_{N}-2),{\rm TW}_{\beta})\leq N^{-\frac{2}{9}+c}.
$$
\end{theorem}

As another illustration of the method described in subsection \ref{sub:sketch}, we derive new typical eigenvalue fluctuations, close to the edge of the spectrum. In the result below and along the paper, we define the typical location $\gamma_k$ of the $k$-th ordered eigenvalue implicitly through
\begin{equation}
\int_{-\infty}^{\gamma_k}\rd\rho=\frac{k}{N}\label{eqn:gamma}.
\end{equation}
\begin{theorem}[Eigenvalues fluctuations close to the edge]\label{Gaussian}
Let $(H_N)$ be generalized Wigner matrices  satisfying (\ref{eqn:tail}). Consider
$$
X_i=c\ \frac{\lambda_{i}-\gamma_{i}}{(\log i)^{1/2}N^{-2/3}i^{-1/3}},
$$
where $c=(3/2)^{1/3}\pi \beta^{1/2}$, with $\beta=1$ for the symmetric class, $2$ for the Hermitian one. 
Fix    $\delta \in (0, 1)$.  Then for any deterministic sequence $i=i_N\to\infty$, with $i\leq N^\delta$,
we have $X_i\to \mathcal{N}(0,1)$ in distribution.

Let $m\geq 1$ and 
 $k_1<\dots<k_m$ satisfy $k_1\sim N^\delta$,
$k_{i+1}-k_i\sim N^{\vartheta_i}$, $0<\vartheta_i\leq\delta$.
Then $(X_{k_1},\dots,X_{k_m})$ converges to a Gaussian vector with covariance matrix
$\Lambda_{ij}=1-\delta^{-1}\max\{\vartheta_k,i\leq k<j\}$ if $i<j$, $\Lambda_{ii}=1$.
\end{theorem}

These anomalous small Gaussian fluctuations were first shown in \cite{Gus2005} for GUE and \cite{Oro2010} for GOE. Our proof proceeds by comparison with these results.
Fluctuations of eigenvalues around their typical locations are known in the bulk of the spectrum for Wigner matrices \cite{BouMod2018,LanSos2018}. 
Theorem \ref{Gaussian} extends to any $\delta\in(0,1)$ a previous result from \cite{BouErdYau2014} which was limited to $\delta<1/4$, and
therefore completes the proof of eigenvalues' fluctuations anywhere in the spectrum\footnote{
The results of \cite{BouMod2018,LanSos2018} are stated for eigenvalues in $[-2+\kappa,2-\kappa]$, but the proofs immediately extend to 
$[-2+N^{-c},2-N^{-c}]$ for some fixed, small enough $c>0$.
}.

More generally, the proof sketch below explains edge statistics for general observables of eigenvalues with indices in $\llbracket 1, N^{1-\e}\rrbracket$, i.e. almost up to the bulk. As another example, for any fixed $\e>0$ and diverging $i<N^{1-\e}$, $N^{2/3}i^{-1/3}(\lambda_{i+1}-\lambda_i)$ converges to the Gaudin distribution, a result proved in \cite{BouErdYau2014} for $i<N^{1/4}$.

\subsection{Sketch of the proof.}\  \label{sub:sketch}
In this paper we denote  $c$, $C$ generic  small and large constants which do not depend on $N$ but may vary from line to line.
Let $\kappa(z)=\min(|z-2|,|z+2|)$ and 
\begin{equation}\label{eqn:varphi}\varphi=e^{C_0(\log\log N)^2},\end{equation}
a subpolynomial error parameter, for some fixed $C_0>0$. This constant $C_0$ is chosen large enough so that the eigenvalues' rigidity from Lemma \ref{rig} holds.

Finally, we restrict the following outline and the full proof to the symmetric class, the Hermitian one requiring only changes in notations.\\

 As already mentioned, our work
proceeds by interpolation with the integrable models, following the general method from \cite{ErdSchYau2011}.
This dynamic approach requires (i) a priori bounds on the eigenvalues' locations, (ii) local relaxation for the eigenvalues' dynamics after a short time, (iii) a density argument based on the matrix structure, to show that eigenvalues statistics have not changed after short-time dynamics.

In this paper,  (i) is the rigidity estimate from \cite{ErdYauYin2012}.  Concerning the density argument (iii), for theorems \ref{thm:rate} and \ref{Gaussian} we follow the Lindeberg exchange method \cite{TaoVu2011} for Green's functions \cite{ErdYauYin2012Univ}. For theorems \ref{thm:smallprocess} and \ref{thm:largest}, (iii) is obtained through the inverse heat flow from \cite{ErdSchYau2011} (this is where smoothness  is required).

Our contribution is about (ii).
Previous approaches for local convergence to equilibrium included the local relaxation flow based on relative entropy \cite{ErdSchYau2011}.
It identifies eigenvalues statistics after a spatial averaging and therefore does not apply to extrema. Other
methods based either on H\"{o}lder regularity a la Di-Giorgi-Nash-Moser \cite{ErdYau2015} or 
${\rm L}^2$-estimates and a discrete Sobolev inequality \cite{LanSosYau2016} apply to individual eigenvalues but 
give non-explicit error terms. In this paper, we give another approach based on the maximum principle. 
Our main results  are Theorem \ref{thm:edgerelaxation} for relaxation at the edge, and 
Theorem \ref{thm:average} for relaxation in  the bulk. They give the first explicit (and optimal) error estimates for local relaxation of eigenvalues dynamics.\\

The Dyson Brownian motion dynamics are defined as follows.
Let $B$ be a $N\times N$ matrix such that $B_{ij} (i<j)$ and $B_{ii}/\sqrt{2}$ are independent standard Brownian motions, and $B_{ij}=B_{ji}$.
Consider the matrix Ornstein-Uhlenbeck process
$$
\rd H_{t}=\frac{1}{\sqrt{N}} \rd B_t-\frac{1}{2}H_t\rd t.
$$
If  $\lambda_1(0)<\dots<\lambda_N(0)$, the eigenvalues $\bla(t)$ of $H_t$ are given by the strong solution
of the system of stochastic differential equations \cite{Dys1962} (the  $\beta_k$'s are some Brownian motions distributed as the $B_{kk}$'s)
\begin{align*}
\rd\la_k=\frac{\rd \beta_{k}}{\sqrt{N}}+\left(\frac{1}{N}\sum_{\ell\neq k}\frac{1}{\la_k-\la_\ell}-\frac{1}{2}\lambda_k\right)\rd t.
\end{align*}
The coupling method introduced in \cite{BouErdYauYin2016} proceeds as follows. Consider $\bmu(t)$ the solution of the same SDE as above
with another initial condition $\bmu(0)=\{\mu_1(0)<\dots<\mu_N(0)\}$, the spectrum of a GOE matrix. Then the differences $\delta_k(t)=e^{t/2}(\la_k(t)-\mu_k(t))$ satisfy the  long-range parabolic differential equation
$$
\partial_t\delta_k(t)=\sum_{j\neq k}\frac{\delta_j(t)-\delta_k(t)}{N(\la_k(t)-\la_j(t))(\mu_k(t)-\mu_j(t))}.
$$
Smoothing of this equation for indices in the bulk means that for $t\gg 1/N$,
$$
\delta_{k+1}(t)=\delta_k(t)+\oo\left(N^{-1}\right).
$$
Such estimates were proved in \cite{ErdYau2015,LanSosYau2016}, with a weak error term $N^{-1-\e}$ with some non-explicit $\e>0$. We obtain the essentially optimal  estimate (see Corollary \ref{cor:gaps}), up to subpolynomial orders,
\begin{equation}\label{quant}\delta_{k+1}(t)=\delta_k(t)+\OO\left(\frac{\varphi^C}{N^2t}\right).
\end{equation}
With this quantitative relaxation, $(\la_{k+1}-\la_k)-(\mu_{k+1}-\mu_k)$ is below the expected scale of smallest gaps provided $t\gg N^{-1/2}$ for $\beta=1$, $t\gg N^{-2/3}$ for $\beta=2$. This gives the relaxation step (ii) for the smallest gaps. The proof for the large gaps proceeds identically and only requires $t\gg 1/N$.

Our proof of  (\ref{quant}) reduces H\"{o}lder regularity to an elementary maximum principle, and it also applies to edge universality.
In details,  for any $\nu\in[0,1]$, let
\begin{equation}\label{eqn:initial}
x^{(\nu)}_k(0)=\nu \mu_k(0)+(1-\nu) \la_k(0)
\end{equation}
be interpolating between the Wigner and GOE initial conditions, as in \cite{LanSosYau2016}.
Define 
\begin{equation}\label{eqn:dbm}
\rd x^{(\nu)}_k(t)=\frac{\rd \beta_{k}(t)}{\sqrt{N}}+\left(\frac{1}{N}\sum_{\ell\neq k}\frac{1}{x^{(\nu)}_k(t)-x^{(\nu)}_\ell(t)}-\frac{1}{2}x^{(\nu)}_k(t)\right)\rd t.
\end{equation}
Then $\frak{u}^{(\nu)}_k(t)=e^{t/2}\frac{\rd}{\rd \nu}x^{{(\nu)}}_k(t)$ satisfies the non-local parabolic differential equation
\begin{align}\label{eqn:gene}
\frac{\rd}{\rd t}\frak{u}^{(\nu)}_k(t)&=(\mathscr{B}\frak{u}^{(\nu)})(k),
\end{align}
where 
\begin{equation}\label{eqn:coef}
(\mathscr{B} f)(k)=(\mathscr{B}(t) f_t)(k)=\sum_{\ell\neq k} c_{\ell k}(t)(f_t(\ell)-f_t(k)),\ 
 c_{\ell k}(t)=\frac{1}{N(x_k^{(\nu)}(t)-x_\ell^{(\nu)}(t))^2}.
\end{equation}
From now we set $\nu\in(0,1)$ and generally omit it from the notations. Let
\begin{equation}\label{eqn:observable}
f_t(z)=e^{-\frac{t}{2}}\sum_{k=1}^N\frac{\frak{u}_k(t)}{x_k(t)-z}.
\end{equation}
The above function is the main idea in our work. 
Note that for $k$ in the bulk $\frak{u}_k(0)$ is of order $N^{-1}$ so that,
for $\im z$  in the bulk of the spectrum,
$f_0$
is a function of order $1$. From (\ref{eqn:ft}) below, $f_t$ is of the same order.

A key observation is that the quadratic singularities from the denominator in (\ref{eqn:coef}) disappear when combined with the Dyson Brownian motion evolution itself,
so that the time evolution of $f$ has no shocks. This is reminiscent of a similar argument in \cite[Lemma 6.2]{BouYau2017}, for a different observable. More precisely, $f$ follows dynamics close to the advection equation

\begin{wrapfigure}[12]{l}{0.45\textwidth}
\vspace{-0.7cm}
\begin{center}
\includegraphics[width=0.45\textwidth]{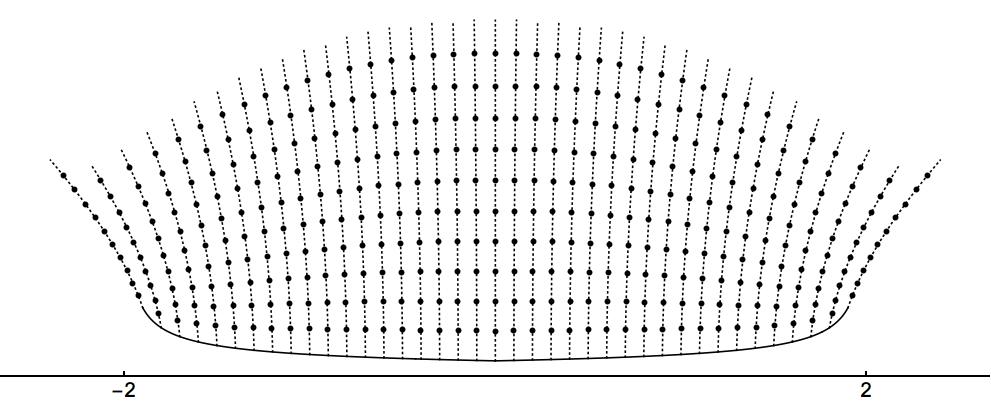}
\caption{The characteristics for the equation (\ref{eqn:adv}), i.e. trajectories $(z_t)_{t\geq 0}$, with $z_0$ on the lower curve $\mathscr{S}$ from $(\ref{eqn:S})$.}\end{center}
\end{wrapfigure}
~~\begin{equation}\label{eqn:adv}
\partial_t h=\frac{\sqrt{z^2-4}}{2}\partial_z h,
\end{equation}
as shown in Lemma \ref{lem:key}.
The charateristics for the above equation are explicit,
\begin{equation}\label{eqn:zteq}
z_t=\frac{e^{t/2}(z+\sqrt{z^2-4})+e^{-t/2}(z-\sqrt{z^2-4})}{2}
\end{equation}
and suggest the approximation
\begin{equation}\label{eqn:ft}
f_t(z)\approx f_0(z_t).
\end{equation}
This estimate holds with a small error term  (see e.g. Proposition \ref{prop:estimatebulk}) because there are no possible shocks between eigenvalues in the equation guiding $f$, contrary to (\ref{eqn:gene}).
The approximation (\ref{eqn:ft}) has two applications.\\

\noindent{\it First application: relaxation at the edge.}\ 
Let ${\frak v}_k={\frak v}_k^{(\nu)}$ solve the same equation as $(\ref{eqn:gene})$ 
($\frac{\rd}{\rd t}{\frak v}_k(t)=(\mathscr{B} {\frak v})(k)$) but with initial condition ${\frak v}_k(0)=|\frak{u}_k(0)|=|\mu_k(0)-\la_k(0)|$.
Similarly to (\ref{eqn:observable}), define 
\begin{equation}\label{eqn:g}
\wt f_t(z)=e^{-\frac{t}{2}}\sum_{k=1}^N\frac{{\frak v}_k(t)}{x_k(t)-z}.
\end{equation}
Edge universality follows from the shape of the characteristics (\ref{eqn:zteq}), 
which take points around the edge further away from the bulk, as shown in Figure 1.
More precisely, we choose $z=z_0=E+\ii \eta$ with $E\in[-2,0]$ and $\eta>0$. 
By a straightforward calculation based on the explicit formula (\ref{eqn:zteq}) and eigenvalues' rigidity, we have
$
\im \wt f_0(z_t)< \varphi^C \frac{\kappa(z_0)^{1/2}}{t}.
$
Together with the estimate (\ref{eqn:ft}) for $\wt f$, we obtain $
\im \wt f_t(z_0)< \varphi^C \frac{\kappa(z_0)^{1/2}}{t}.
$
For $z_0=-2+\ii \varphi^C N^{-2/3}$, as ${\frak v}$ remains nonnegative, this implies
$$
{\frak v}^{(\nu)}_1(t)<\varphi^C  N^{-2/3} \im \wt f_t(z_0)< \frac{\varphi^C}{Nt}.
$$
In particular, integrating the above equation in $0\leq \nu\leq 1$ after using $|{\frak u}_k(t)|\leq{\frak v}_k(t)$ (the linear equation (\ref{eqn:gene}) preserves order of the initial conditions because $c_{\ell k}\geq 0$), we obtain 
$$
\lambda_1(t)-\mu_1(t)=\OO\left(\frac{\varphi^C}{Nt}\right).
$$
Local edge relaxation is therefore proved for any $t> \varphi^C N^{-1/3}$, with an optimal error term.
Such quantitative bounds can be similarly extended to any $\lambda_k(t)-\mu_k(t)$ provided $k\leq N^{1-\e}$ and 
$t> \varphi^C k^{1/3}N^{-1/3}$.
Theorems \ref{thm:rate} and \ref{Gaussian} follow from these relaxation estimates and
a Green function comparison, following \cite{ErdYauYin2012Univ}.\\

\noindent{\it Second application: relaxation in the bulk.}\ 
We now directly work with $f$ instead of $\wt f$. 
Fix some times $u< t$ such that $|u-t|\ll t$, a length scale 
$r\ll t$ and a bulk index $k$. We are 
interested in evaluating ${\frak u}_i(t)$ for $|i-k|\leq N r$.
Assume that for any $s\in[u,t]$ the maximum value of $\frak{u}(s)$ occurs at some index $j=j(s)$ with $|j-k|\leq N r$ (this is generally wrong but the conclusion will remain thanks to a finite speed  of propagation estimate from \cite{BouYau2017}).
We follow the maximum principle as in the analysis of  the eigenvector moment flow from \cite{BouYau2017}: for any 
$\eta>0$ to be chosen, denoting $z=x_j(s)+\ii\eta$, from \ref{eqn:gene}) and the fact that $\frak{u}_j(s)\geq \frak{u}_\ell(s)$ for all $\ell$, we have
$$
\partial_s \frak{u}_j(s)\leq \frac{1}{N}\sum_{\ell\neq j} \frac{\frak{u}_\ell(s)-\frak{u}_j(s)}{(x_\ell(s)-x_j(s))^2+\eta^2}
\leq \frac{c}{\eta} \left(\frac{1}{N}\im f_s(z)-\im \left(\frac{1}{N}\sum_{\ell}\frac{1}{z-x_\ell(s)}\right)\frak{u}_j(s)\right).
$$
In the bulk of the spectrum, (\ref{eqn:ft})  holds with the good error term $\varphi^C/(N\eta)$ (see Proposition \ref{prop:estimatebulk}), so that the previous equation behaves similarly to ($\im \left(\frac{1}{N}\sum_{\ell}\frac{1}{z-x_\ell(s)}\right)\approx\im m_{\rm sc}(z)$ by eigenvalues' rigidity, where $m_{\rm sc}$ is defined in (\ref{eqn:st}))
$$
\partial_s \frak{u}_j(s)\leq \frac{c}{\eta}\left(\frac{1}{N}\im f_0(z_s)-\im m_{\rm sc}(z)\frak{u}_j(s)\right)+\OO\left(\frac{\varphi^C}{N^2\eta^2}\right).
$$
We can successively justify and quantify the approximations 
$z_s=(x_j(s)+\ii\eta)_s\approx (\gamma_j+\ii\eta)_s\approx(\gamma_k+\ii\eta)_s\approx(\gamma_k+\ii\eta)_t$
by rigidity of the eigenvalues, $r\ll s$ and $s\approx t$. 
We therefore can substitute
 $\im f_0(z_s)\approx
\im f_0((\gamma_{k}+\ii\eta)_{t})
$, so that
denoting $m_t=\im f_0((\gamma_{k}+\ii\eta)_{t})/(N \im m_{\rm sc}(\gamma_{k}+\ii 0^+))$,
the above equation implies
$$
\partial_s (\frak{u}_j(s)-m_t)\leq -\frac{c}{\eta}(\frak{u}_j(s)-m_t)+\OO\left(\frac{\varphi^C}{N^2\eta^2}\right).
$$
For any $\eta\ll |t-u|$ we obtain $\max_{|i-k|\leq Nr} (\frak{u}_i(t)-m_t)=\OO(\frac{\varphi^C}{N^2\eta})$.
The same estimate naturally holds for the minimum.
If the time evolution  $|t-u|$ is comparable to $t$, we obtain $\max_{|i-k|\leq Nr} |\frak{u}_i(t)-m_t|=\OO(\frac{\varphi^C}{N^2t})$, and in particular
\begin{equation}\label{eqn:ukfin}
\frak{u}_{k+1}^{(\nu)}(t)=\frak{u}_k^{(\nu)}(t)+\OO\left(\frac{\varphi^C}{N^2t}\right).
\end{equation}
The above argument is rigorous up to some technicalities due to localizing the maximum in the window $|i-k|\leq Nr$.
The actual proof proceeds by induction in different space-time windows. The key to make this maximum principle work is that $f_s(z)$ (possibly highly oscillatory in the space variable $\re z$), actually fluctuates on a
large scale  thanks to (\ref{eqn:ft}), and can be considered constant in windows of size $r\ll t$.

Integrating (\ref{eqn:ukfin}) over $\nu\in(0,1)$, we obtain  (\ref{quant}), which is the main estimate for theorems \ref{thm:smallprocess} and \ref{thm:largest}.\\

To summarize this proof sketch, the observable (\ref{eqn:observable}) and the stochastic advection equation (\ref{eqn:adv}) it satisfies are new ingredients  to quantify relaxation of the Dyson Brownian motion and obtain universality beyond microscopic scales. 

It has been known  since \cite{Pas1972} that a deterministic advection equation allows to derive the semicircle distribution.
More recent works (e.g. \cite{AllBunBou2014,FacVivBir2016,SooWar2018}) have written the stochastic advection equation for the resolvent of a matrix following the Dyson Brownian motion dynamics.

These resolvent dynamics can be used for regularization and universality purpose, as proved first in
\cite{LeeSch2015}, for eigenvalues statistics at the edge of deformed Wigner matrices. For the same model,  \cite{Ben2017,SooWar2019} used stochastic advection equations and characteristics to understand the shape of individual bulk eigenvectors. Moreover,  the stochastic advection equation for the Stieltjes transform extends to general $\beta$-ensembles and allows to prove rigidity of the particles \cite{HuaLan2019,AdhHua2018}, also through regularization along the characteristics.
The Stieltjes transform is a specialization of our observable $f_t$ when
 $\frak{u}_k(t)\equiv \frac{1}{N}$.\\

\noindent{\bf Acknowledgement.} The idea developed in this paper benefited from discussions with students of graduate classes at the Courant Institute in 2015, 2018,  the Saint Flour summer school in 2016 and the IHES summer school in 2017.
The author also thanks the organizers and participants of the workshop \cite{Open2010} where the questions of universality of extreme gaps, and rate of convergence in universality, were raised.
Finally, the author warmly thanks Gaultier Lambert
and
Patrick Lopatto,
whose detailed comments improved the manuscript.
This work is partially supported by the Poincar\'e chair and the NSF grant DMS-1812114.

\section{Stochastic advection equation}
\label{sec:stoch}

\subsection{The observable.}\ 
The Stieltjes transform of the empirical spectral measure and the semicircle law $\rho(x)=\frac{1}{2\pi}\sqrt{(4-x^2)_+}$ are denoted 
\begin{equation}\label{eqn:st}s_t(z)=\frac{1}{N}\sum_{k=1}^N\frac{1}{{x_k(t)}-z},\ 
m(z)=m_{\rm sc}(z)=\int_{\mathbb{R}}\frac{\rd \rho(x)}{x-z}=\frac{-z+\sqrt{z^2-4}}{2},
\end{equation}
where our branch choice will always be ${\rm Im}\sqrt{z^2-4}>0$ for ${\rm Im}(z)>0$, above and in (\ref{eqn:zteq}).

More generally than (\ref{eqn:dbm}), consider $\bx(t)$ the strong solution of
\begin{equation}\label{eqn:beta}
\rd x_k(t)=\frac{\sqrt{2}\rd B_{k}(t)}{\sqrt{\beta N}}+\left(\frac{1}{N}\sum_{\ell\neq k}\frac{1}{x_k(t)-x_\ell(t)}-\frac{1}{2}x_k(t)\right)\rd t,
\end{equation}
where the $B_k$'s are standard Brownian motions, $\bx(0)$ is still given by (\ref{eqn:initial}), and $\beta=1$ (resp. $\beta=2$) corresponds to the spectral dynamics with equilibrium measure GOE (resp. GUE). For any $\beta\geq 1$
and distinct initial points, the stochastic differential equation (\ref{eqn:beta}) admits a unique strong solution.

We still define $\frak{u}^{(\nu)}_k(t)=e^{t/2}\frac{\rd}{\rd \nu}x^{{(\nu)}}_k(t)$. Then the function (\ref{eqn:observable}) satisfies the following dynamics.

\begin{lemma}\label{lem:key}
For any $\im z\neq 0$, we have
\begin{equation}\label{eqn:dynamicsPDE}
\rd f_t=\left(s_t(z)+\frac{z}{2}\right)(\partial_z f_t)\rd t
+\frac{1}{N}\left(\frac{2}{\beta}-1\right)(\partial_{zz}f_t)\rd t
-\frac{e^{-t/2}}{\sqrt{N}}\sqrt{\frac{2}{\beta}}\sum_{k=1}^N\frac{\frak{u}_k(t)}{(z-x_k(t))^2}\rd B_k(t).
\end{equation}
\end{lemma}

\begin{proof} It is a simple application of It{\^ o}'s formula. We omit the time index. First,
\begin{equation}\label{eqn:Ito1}
\rd f = -\frac{f}{2}+e^{-\frac{t}{2}}\sum_{k=1}^N\frac{\rd \frak{u}_k}{x_k-z}+e^{-\frac{t}{2}}\sum_{k=1}^N\frak{u}_k\rd\frac{1}{x_k-z}.
\end{equation}
Applying again the It{\^ o} formula $\rd (x_k-z)^{-1}=-(x_k-z)^{-2}\rd x_k+\frac{2}{\beta N}(x_k-z)^{-3}\rd t$, with (\ref{eqn:beta}) we naturally decompose the second sum above as (I)+[(II)+(III)+(IV)]$\rd t$ where
\begin{align*}
{\rm (I)}&=-\frac{e^{-t/2}}{\sqrt{N}}\sqrt{\frac{2}{\beta}}\sum_{k=1}^N\frac{\frak{u}_k}{(z-x_k)^2}\rd B_k,\\
{\rm (II)}&=\frac{e^{-t/2}}{N}\sum_{\ell\neq k}\frac{\frak{u}_k}{x_\ell-x_k}\frac{1}{(x_k-z)^2},\\
{\rm (III)}&=\frac{e^{-t/2}}{2}\sum_{k=1}^N\frac{\frak{u}_k x_k}{(x_k-z)^2},
=\frac{f}{2}+\frac{z}{2}\partial_z f,
\\
{\rm (IV)}&=\frac{2e^{-t/2}}{\beta N}\sum_{k=1}^N \frac{\frak{u}_k}{(x_k-z)^3}=\frac{1}{N}\left(\frac{2}{\beta}-1\right)
\partial_{zz}f
+\frac{e^{-t/2}}{N}\sum_{k=1}^N \frac{\frak{u}_k}{(x_k-z)^3}.
\end{align*}
Concerning the first sum in (\ref{eqn:Ito1}), by (\ref{eqn:gene}) we have
\begin{align*}
\sum_{k=1}^N\frac{\partial_t \frak{u}_k}{x_k-z}
&=\sum_{\ell\neq k}\frac{\frak{u}_\ell-\frak{u}_k}{N(x_\ell-x_k)^2(x_k-z)}
=\frac{1}{2}\sum_{\ell\neq k}\frac{\frak{u}_\ell-\frak{u}_k}{N(x_\ell-x_k)^2}\left(\frac{1}{x_k-z}-\frac{1}{x_\ell-z}\right)
\\
&=\frac{1}{2N}\sum_{\ell\neq k}\frac{\frak{u}_\ell-\frak{u}_k}{x_\ell-x_k}\frac{1}{(x_k-z)(x_\ell-z)}
=-\frac{1}{N}\sum_{\ell\neq k}\frac{\frak{u}_k}{x_\ell-x_k}\frac{1}{(x_k-z)(x_\ell-z)}.
\end{align*}
Combining with (II), we obtain 
$$
{\rm (II)}+e^{-t/2}\sum_{k=1}^N\frac{\partial_t \frak{u}_k}{x_k-z}=
\frac{e^{-t/2}}{N}\sum_{\ell\neq k}\frac{\frak{u}_k}{x_\ell-x_k}\frac{1}{x_k-z}\left(\frac{1}{x_k-z}-\frac{1}{x_\ell-z}\right)
=
\frac{e^{-t/2}}{N}\sum_{\ell\neq k}\frac{\frak{u}_k}{(x_k-z)^2}\frac{1}{x_\ell-z}.
$$
All singularities have disappeared. We obtained
$
{\rm (II)}+{\rm (IV)}+e^{-t/2}\sum_{k=1}^N\frac{\partial_t \frak{u}_k}{x_k-z}
=s(z)\partial_z f+\frac{1}{N}\left(\frac{2}{\beta}-1\right)
\partial_{zz}f.
$
Summation of the remaining terms (I) and (III) concludes the proof.
\end{proof}

Remember that $\kappa(z)=\min\{|z-2|,|z+2|\}$, and define
$$
a(z)={\rm dist}(z,[-2,2]),\ b(z)={\rm dist}(z,[-2,2]^c).
$$
To estimate $f_t$ or $\wt f_t$ (see (\ref{eqn:g})), we first need some bounds on the characteristics $(z_t)_{t\geq 0}$ from (\ref{eqn:zteq}),  and the initial values $f_0$, $\wt f_0$.
For this, we define the curve
\begin{equation}\label{eqn:S}
\mathscr{S}=
\left\{
z= E+\ii y:
-2+\varphi^2 N^{-2/3}<E<2-\varphi^2 N^{-2/3},
y=\varphi^2/(N\kappa(E)^{1/2})
\right\}
\end{equation}
and the domain $\mathscr{R}=\cup_{0\leq t\leq 1}\{z_t:z\in \mathscr{S}\}$.
See Figure 1 for a representation of these domains.

In the following lemma, we denote $a\sim b$ if there exists $C>0$ such that $C^{-1}b<a<C b$ for all specified parameters
$z,t$. For complex valued functions, $a\sim b$ means $\re a\sim\re b$ and $\im a\sim \im b$.

\begin{lemma}\label{lem:zt}
Uniformly in $0<t<1$ and $z=z_0$ satisfying $\eta=\im z>0$, $|z-2|<1/10$, we have
$$
\re (z_t-z_0)\sim t \frac{a(z)}{\kappa(z)^{1/2}}+t^2,\ \ \ 
\im (z_t-z_0)\sim \frac{b(z)}{\kappa(z)^{1/2}}t.
$$
In particular, if in addition we have $z\in\mathscr{S}$, then
$
z_t-z_0\sim \left(t \frac{\varphi^2}{N\kappa(E)}+t^2\right)+\ii\kappa(E)^{1/2}t.
$

Moreover,  for any $\kappa>0$, uniformly in $0<t<1$ and $z=E+\ii\eta\in[-2+\kappa,2-\kappa]\times[0,\kappa^{-1}]$,
we have $\im (z_t-z_0)\sim t$.
\end{lemma}
\begin{proof}
Let $w=z-2$. We have $(z^2-4)^{1/2}\sim w^{1/2}$ so that 
\begin{align}
&\re (z^2-4)^{1/2}\sim \re (w^{1/2})\sim |w|^{1/2}\re((w/|w|)^{1/2})\sim \frac{a(z)}{\kappa(z)^{1/2}},\notag\\
&\im (z^2-4)^{1/2}\sim \im(w^{1/2})\sim |w|^{1/2}\im((w/|w|)^{1/2})\sim \frac{b(z)}{\kappa(z)^{1/2}}\label{eqn:imagM}.
\end{align}
On $\mathscr{S}$, we always have $b(z)\sim \kappa(z)$ and $a(z)\sim \eta$ so the second estimate follows immediately.

The last estimate follows from $\im \sqrt{z^2-4}\sim 1$ uniformly in the defined bulk domain.
\end{proof}

We now define the typical eigenvalues' location and the set of good trajectories such that rigidity holds:
\begin{align}
&\mathscr{A}=
\left\{  |x^{(\nu)}_k(t)-\gamma_k|<\varphi^{1/2} N^{-\frac{2}{3}}(\hat k)^{-\frac{1}{3}}\ {\rm for\, all}\ 0\leq t\leq1,  k\in\llbracket 1,N\rrbracket, 0\leq \nu\leq 1\right\}\label{eqn:A},
\end{align}
where $\hat k=\min(k,N+1-k)$.
The following important a priori estimates were proved in \cite{ErdYauYin2012}, for fixed $t$ and $\nu=0$ or 1.
The extension in these parameter is straightforward, by time discretization in $t$ and $\nu$ first, then by Weyl's inequality
to bound increments in small time intervals, and the fact that $|{\frak u}^{(\nu)}_k(t)|<\|{\frak u}^{(\nu)}(0)\|_\infty$ to bound increments in some small $\nu$-intervals.

\begin{lemma}\label{rig} There exists a fixed $C_0>0$ (remember $\varphi=\varphi(C_0)$) large enough such that the following holds.
For any $D>0$, there exists $N_0(D)$ such that for any $N>N_0$ we have
$$
\mathbb{P}(\mathscr{A})>1-N^{-D}.
$$
\end{lemma}

Moreover, we have the following estimates on the initial condition $f_0,\wt f_0$.

\begin{lemma}\label{f0} 
In the set $\mathscr{A}$, for any $z=E+\ii\eta\in \mathscr{R}$, we have
$
\im \wt f_0(z)\leq C \varphi^{1/2}
$
if $\eta>\max(E-2,-E-2)$, and $
\im \wt f_0(z)\leq C \varphi^{1/2} \frac{\eta}{\kappa(z)} 
$ otherwise.
The same upper bound naturally holds for $|\im f_0|$.
\end{lemma}

\begin{proof}
The rigidity estimate on $\mathscr{A}$ easily implies that
$$
\im \wt f_0(z)\leq C \eta \sum_{k=1}^N\frac{ \varphi^{1/2} N^{-2/3}(\hat k)^{-1/3}}{(E-\gamma_k)^2+\eta^2}\leq C \varphi^{1/2}\eta\int_{-2}^2
\frac{\kappa(x)^{-1/2}}{(E-x)^2+\eta^2}\rd\rho(x)\leq C \varphi^{1/2}\eta
\int_{-2}^2
\frac{1}{(E-x)^2+\eta^2}\rd x,
$$
and the claimed estimates follow. Note that we used $z\in \mathscr{R}$ to justify approximation of eigenvalues by typical location: in $\mathscr{R}$ the imaginary part of $z$ is always greater than the eigenvalues' fluctuation scale.
\end{proof}

Finally, the following is an elementary calculation. We write
$z_t=r(z,t)$, for $r$ given by the right-hand side of (\ref{eqn:zteq}). 

\begin{lemma}\label{eqn:char}
We have $\partial_t r=\frac{\sqrt{z^2-4}}{2}\partial_z r$.
\end{lemma}

\subsection{Relaxation at the edge.}\ 
For the following important estimate towards edge universality, remember the notation (\ref{eqn:g}).

\begin{proposition}\label{prop:estimate}
Consider the dynamics (\ref{eqn:dynamicsPDE}) for $\beta=1,2$.  For any (large) $D>0$ there exists $N_0(D)$ such that for 
any $N\geq N_0$ we have
$$
\mathbb{P}\left(\im \wt f_t(z)\leq \varphi\frac{\kappa(E)^{1/2}}{\max(\kappa(E)^{1/2},t)}\ {\rm for\ all}\ 0<t<1\ {\rm and}\ z=E+\ii y\in\mathscr{S}\right)>1-N^{-D}.
$$
\end{proposition}

\begin{proof}

For any $1\leq \ell,m\leq N^{12}$, we define $t_\ell=\ell N^{-12}$ and $z^{(m)}=E_m+\ii\eta_m=E_m+\ii\frac{\varphi^2}{N\kappa(E_m)^{1/2}}$ where $\int_{-{\infty}}^{E_m}\rd\rho=(m-1/2) N^{-12}$.
We also define the stopping times (with respect to $\mathcal{F}_t=\sigma(B_k(s),0\leq s\leq t,1\leq k\leq N)$)
\begin{align}
\tau_{\ell,m}&=\notag
\inf\left\{0\leq s\leq t_\ell:  \im \wt f_s(z^{(m)}_{t_\ell-s})>\frac{\varphi}{2}\frac{\kappa(E_m)^{1/2}}{\max(\kappa(E_m)^{1/2},t_\ell)}\right\},\\
\tau_0&=\label{eqn:t0}
\inf\left\{0\leq t\leq 1\mid \exists k\in\llbracket 1,N\rrbracket : |x_k(t)-\gamma_k|>\varphi^{1/2} N^{-\frac{2}{3}}(\hat k)^{-\frac{1}{3}}\right\},\\
\tau&=\min\{\tau_0,\tau_{\ell,m}:0\leq \ell,m\leq N^{12}, \kappa(E_m)>\varphi^2 N^{-2/3}\}\notag,
\end{align}
with the convention $\inf\varnothing=1$.
We will prove that for any $D>0$ there exists $\wt N_0(D)$ such that for any $N\geq \wt N_0(D)$, we have 
\begin{equation}\label{eqn:inter1}
\mathbb{P}(\tau=1)>1-N^{-D}.
\end{equation}
We first explain why the above equation implies the expected result by a grid argument in $t$ and $z$.

On the one hand, we have the sets inclusion
\begin{equation}\label{eqn:inter2}
\{\tau=1\}\bigcap_{1\leq \ell,m\leq N^{12},1\leq k\leq N}A_{\ell,m,k}
\subset
\bigcap_{z\in \mathscr{S},0<t<1}\left\{\im \wt f_t(z)\leq \varphi\frac{\kappa(E)^{1/2}}{\max(\kappa(E)^{1/2},t)}\right\}
\end{equation}
where
$$
A_{\ell,m,k}=\left\{
\sup_{t_\ell\leq u\leq t_{\ell+1}}\left|\int_{t_\ell}^u
\frac{e^{-s/2}{\frak v}_k(s)\rd B_k(s)}{(z^{(m)}-x_k(s))^2}
\right|<N^{-4}
\right\}.
$$
Indeed, for any given $z$ and $t$, chose $t_\ell,z^{(m)}$ such that $t_\ell\leq t<t_{\ell+1}$ and $|z-z_m|<N^{-5}$. Then
$
|\wt f_t(z)-\wt f_t(z^{(m)})|<N^{-2},
$
say, as follows directly from the definition of $\wt f_t$ and the crude estimate $|{\frak v}_k(t)|<1$ (obtained by maximum principle). 
Moreover, we can bound the time increments using  (\ref{eqn:dynamicsPDE}): Thanks to 
the trivial estimates 
$|s_t(E+\ii\eta)|\leq \eta^{-1}$, 
$|\partial_z f_t(E+\ii\eta)|
\leq N\|\frak{v}(0)\|_\infty\eta^{-2}\leq N\eta^{-2}$ and  $|\partial_{zz} f_t(E+\ii\eta)|\leq N\eta^{-3}$, under the event $\cap_kA_{\ell,m,k}$ (to bound the martingale term)
we have
$
|\wt f_t(z^{(m)})-\wt f_{t_\ell}(z^{(m)})|<N^{-2}.
$

On the other hand, from \cite[Appendix B.6, equation (18)]{ShoWel2009} with  $c=0$ allowed for continuous martingales,  for any continuous martingale $M$ and any $\lambda,\mu>0$, we have 
\begin{equation}\label{eqn:bracket}
\mathbb{P}\Big(\sup_{0\leq u\leq t}|M_u|\geq \lambda,\ \langle M\rangle_{t}\leq \mu\Big)\leq 2 e^{-\frac{\lambda^2}{2\mu}}.
\end{equation}
For $M_u=\int_{t_\ell}^u
\frac{e^{-s/2}{\frak v}_k(s)\rd B_k(s)}{(z^{(m)}-x_k(s))^2}$, we have the deterministic estimate
$\langle M\rangle_{t_{\ell+1}}\leq N^{-12}(\varphi^2/N)^{-4}\|\frak{v}(0)\|_\infty^2\leq \varphi^{-8}N^{-8}$,
so that (\ref{eqn:bracket}) with $\mu=\varphi^{-8}N^{-8}$ gives
$
\mathbb{P}(A_{\ell,m,k})\geq 1-e^{-c\varphi^{1/5}}
$
and therefore, for any $D>0$, for large enough $N$ we have
\begin{equation}\label{eqn:inter3}
\mathbb{P}\Big(
\bigcap_{1\leq \ell,m\leq N^{10},1\leq k\leq N}A_{\ell,m,k}
\Big)
\geq 
1-N^{-D}.
\end{equation}
Equations (\ref{eqn:inter1}), (\ref{eqn:inter2}), (\ref{eqn:inter3}) conclude the proof of the proposition.

We now prove (\ref{eqn:inter1}). We abbreviate $t=t_\ell$, $z=E+\ii\eta=z^{(m)}$ for some $1\leq \ell,m\leq N^{10}$. Let $g_u(z)=\wt f_{u}(z_{t-u})$. 
From lemmas \ref{lem:zt} and \ref{f0}, 
$\im g_0(z) \leq \frac{\varphi}{10}\frac{\kappa(E_m)^{1/2}}{\max(\kappa(E_m)^{1/2},t)},$
so that we only need to bound the increment of $g$.
Using lemmas \ref{lem:key} and \ref{eqn:char}, It\^o's formula gives\footnote{In this paper, we abbreviate $u\wedge t=\min(u,t)$ when $u$ and $t$ are time variables.}
\begin{equation}\label{eqn:gev}
\rd g_{u\wedge \tau}(z)=\e_{u}(z_{t-u})\rd({u\wedge \tau})-\frac{e^{-u/2}}{\sqrt{N}}\sqrt{\frac{2}{\beta}}\sum_{k=1}^N\frac{{\frak v}_k({u})}{(z_{t-u}-x_k(u))^2}\rd B_k({u\wedge \tau})
\end{equation}
where
$
\e_{u}(z)=(s_u(z)-m(z))\partial_z \wt f_u+\frac{1}{N}\left(\frac{2}{\beta}-1\right)(\partial_{zz}\wt f_u).
$
We  bound $\sup_{0\leq s\leq t}|\int_0^{s}\e_{u}(z_{t-u})\rd({u\wedge \tau})|$ by two terms, the first one being
\begin{multline}
\int_0^{t}\left| (s_u(z_{t-u})-m(z_{t-u})\partial_z \wt f_{u}(z_{t-u})\right|\rd (u\wedge \tau)
\leq
\int_0^{t}\frac{\varphi}{N\im(z_{t-u})}\sum_{k=1}^N\frac{{\frak v}_k(u)}{|z_{t-u}-x_k(u)|^2}\rd (u\wedge \tau)\\
\leq 
\int_0^{t}\frac{\varphi\, {\rm Im} \wt f_u(z_{t-u})}{N(\im(z_{t-u}))^2}\rd (u\wedge \tau)
\leq 
\int_0^{t}\frac{\varphi^2\,\rd u}{N(\eta+(t-u)\frac{b(z)}{\kappa(z)^{\frac{1}{2}}})^2}\frac{\kappa(E)^{\frac{1}{2}}}{\max(\kappa(E)^{\frac{1}{2}},t)}
=\frac{\kappa(E)^{\frac{1}{2}}}{\max(\kappa(E)^{\frac{1}{2}},t)}.\label{eqn:estim1}
\end{multline}
To bound $s_u-m$ above, we have used the strong local semicircle law from \cite[equation (2.19)]{ErdYauYin2012} simultaneously for all $0\leq u\leq t$  (equivalent to Lemma \ref{rig}). We have then used  Lemma \ref{lem:zt} to evaluate $\im(z_{t-u})$,  $u<\tau_{\ell,m}$ to bound 
$\im \wt f_u(z_{t-u})$, and  $\kappa(E)=\kappa(z)=b(z)$ on $\mathscr{S}$ to calculate  the last integral.

We also have 
\begin{equation}\label{estim2}
\sup_{0\leq s\leq t}\left|\int_0^{s} \frac{1}{N}\partial_{zz}\wt f_u(z_{t-u})\rd (u\wedge \tau)\right|
\leq
\int_0^{t}\frac{{\rm Im} \wt f_u(z_{t-u})}{N(\im(z_{t-u}))^2}\rd (u\wedge\tau)\leq \frac{\kappa(E)^{1/2}}{\varphi\max(\kappa(E)^{1/2},t)}.
\end{equation}

Finally, we want to bound $\sup_{0\leq s\leq t}|M_s|$ where
$$
M_s:=\int_0^{s}\frac{e^{-u/2}}{\sqrt{N}}\sum_{k=1}^N\frac{{\frak v}_k({u})}{(z_{t-u}-x_k(u))^2}\rd B_k(u\wedge\tau).
$$
Note that there is an absolute constant $c>0$ such that for all $k$ and $u\leq \tau_0$ we have $|z_{t-u}-x_k(u)|\geq c|z_{t-u}-\gamma_k(u)|$, because for such $u$ we have $|x_k(u)-\gamma_k(u)|\ll |z_{t-u}-\gamma_k(u)|$. With (\ref{eqn:bracket}) we expect
\begin{equation}\label{eqn:error2}
\sup_{0\leq s\leq t}|M_{s}|^2\leq \varphi^{1/10}
\int_0^{t}\frac{1}{N}\sum_k\frac{{\frak v}_k(u)^2}{|z_{t-u}-\gamma_k|^4}\rd (u\wedge\tau)
\end{equation}
with overwhelming probability. More precisely, we will bound the 
above bracket on the right-hand side by a deterministic bound below, and then (\ref{eqn:bracket}) implies the same bound on the right-hand side.

Let $ k_j=\lfloor j \varphi^2\rfloor$ and $I_j=\llbracket k_j ,k_{j+1}\rrbracket\cap\llbracket1,N\rrbracket$, $0\leq j\leq N/\varphi^2$.
Then
\begin{equation}\label{eqn:boundsup}
\frac{1}{N}\sum_k\frac{{\frak v}_k(u)^2}{|z_{t-u}-\gamma_k|^4}
\leq 
\frac{1}{N}\sum_{0\leq j\leq N/\varphi^2}
\left(\max_{k\in I_j}{\frak v}_k(u)\right)\left(\max_{k\in I_j}\frac{1}{|z_{t-u}-\gamma_k|^4}\right)
\left(\sum_{k\in I_j}{\frak v}_k(u)\right).
\end{equation}
For each $0\leq j\leq N/\varphi^2$, pick a $n=n_j$ such that 
$|z^{(n)}-\gamma_{k_j}|<N^{-9}$. 
First, as ${\frak v}_k(u)\geq 0$ for any $k$ and $u$,  we have  
$
\sum_{k\in I_j}{\frak v}_k(u)\leq \eta_n\im \wt f_u(z^{(n)})
$.
To estimate $\im \wt f_u(z^{(n)})$, introduce $\ell$ such that $t_\ell\leq u<t_{\ell+1}$.
On the event $\cap_kA_{\ell,m,k}$ and $u\leq \tau$, we have
$
|\wt f_{u}(z^{(n)})-\wt f_{t_\ell}(z^{(n)})|<N^{-2}
$
as seen easily  from (\ref{eqn:dynamicsPDE}).
We therefore proved 
$$
\sum_{k\in I_j}{\frak v}_k(u)\leq \eta_n\im \wt f_{ t_\ell}(z^{(n)})+N^{-2}\leq \frac{\varphi^3}{N\max(\kappa(\gamma_{E_n})^{1/2},u)},
$$
and in particular the same estimate holds for $\max_{k\in I_j}v_k(u)$. We used $t_\ell\leq u\leq \tau$ for the second inequality.

Lemma \ref{lem:A1} allows us to bound $\max_{k\in I_j}\frac{1}{|z-\gamma_k|^4}$ in (\ref{eqn:boundsup}) by 
$\varphi^{-2}\sum_{I_j}\frac{1}{|z-\gamma_k|^4}$.

All together, with e obtained
$$
\sup_{0\leq s\leq t}|M_{s}|^2\leq \frac{\varphi^{4+\frac{1}{5}}}{N^2}
\int_0^{t}\rd u\int_{-2}^2\frac{\rd \rho(x)}{|z_{t-u}-x|^4\max(\kappa(x),u^2)}
\leq
C\varphi^{\frac{1}{5}}\frac{\kappa(E)}{\max(\kappa(E),t^2)},
$$
where for the last inequality, we evaluate
this deterministic integral in Lemma \ref{lem:A2}.

In conclusion, by a union bound we have proved that for any $D>0$ there exists $N_0$ such that
$$
\mathbb{P}\left(
\sup_{0\leq \ell,m\leq N^{10}, \kappa(E_m)>\varphi^2 N^{-2/3},0\leq s\leq t_\ell}
\im \wt f_{s\wedge\tau}(z^{(m)}_{t_\ell-s\wedge\tau})>\frac{\varphi}{2}\frac{\kappa(E_m)^{1/2}}{\max(\kappa(E_m)^{1/2},t_\ell)}
\right)<N^{-D}.
$$
Together with Lemma \ref{rig}, this implies (\ref{eqn:inter1}) and the result.
\end{proof}

\begin{corollary}\label{cor:v}
For any $D>0$ there exists $N_0$ and such that for any $N>N_0$, we have 
$$
\mathbb{P}\left(
{\frak v}^{(\nu)}_k(t)<\frac{\varphi^{10}}{N}\frac{1}{\max((\hat k/N)^{1/3},t)}\ {\rm for\ all}\ k\in\llbracket 1,N\rrbracket \ {\rm and}\  t\in[0,1]
\right)
>1-N^{-D}.
$$
\end{corollary}

\begin{proof}
Assume first that $\hat k:=\min(k,N+1-k)>\varphi^5$. Then define
$
z=z_0=\gamma_k+\ii\frac{\varphi^2}{N\sqrt{\kappa(\gamma_k)}}\in\mathscr{S}.
$
On $\mathscr{A}$, we have (we use nonnegativity of the ${\frak v}_k$'s)
$$
{\frak v}^{(\nu)}_k(t)
\leq {\frak v}^{(\nu)}_k(t)\frac{(\im z)^2}{|z-x_k|^2}
\leq \im z \im \wt f_t(z)
\leq \frac{\varphi^2}{N\sqrt{\kappa(\gamma_k)}}\im \wt f_t(z).
$$
Note that $\kappa(\gamma_k)^{1/2}\sim (\hat k/N)^{1/3}$. Therefore, by Lemma \ref{rig} and Proposition \ref{prop:estimate},
$$
\mathbb{P}\left(
{\frak v}^{(\nu)}_k(t)<\frac{\varphi^{4}}{N}\frac{1}{\max((k/N)^{1/3},t)}\ {\rm for\ all}\ k\in\llbracket \varphi^5,N-\varphi^5\rrbracket,\ t\in[0,1]
\right)
>1-N^{-D}.
$$
If $\hat k<\varphi^5$, without loss of generality we assume $k<\varphi^5$. The same reasoning with 
$z=z_0=\gamma_{k_0}+\ii\frac{\varphi^2}{N\sqrt{\kappa(\gamma_{k_0})}}\in\mathscr{S}$ (where $k_0=\varphi^5$)
yields to the same estimate up to the deteriorated $\varphi^{10}$ exponent, say. 
\end{proof}

We now state the quantitative relaxation of the dynamics at the edge. Remember that $\bla$ and $\bmu$ satisfy the same equation (\ref{eqn:dbm}), with respective initial conditions a generalized Wigner and GOE spectrum.

\begin{theorem}\label{thm:edgerelaxation}
Consider the dynamics (\ref{eqn:dbm}) (or its Hermitian ensemble counterpart). For any $D>0$ and $\e>0$ there exists $N_0$ and such that for any $N>N_0$,
$$
\mathbb{P}\left(
|\lambda_k(t)-\mu_k(t)|<\frac{N^{\e}}{Nt}\  {\rm for\ all}\ k\in\llbracket 1,N\rrbracket \ {\rm and}\  t\in[0,1]
\right)
>1-N^{-D}.
$$
\end{theorem}

\begin{remark}\label{rem:anyBeta} The above result is stated for $\beta=1,2$.
The same result holds for any choice $\beta\geq 1$ in the equation (\ref{eqn:beta}), provided $\bla$ and $\bmu$ satisfy optimal initial rigidity estimates. The proof only requires notational changes.
\end{remark}

\begin{proof}
Remember that ${\frak v}-{\frak u}$ and ${\frak v}+{\frak u}$ are nonnegative for $t=0$ and satisfies the equation (\ref{eqn:gene}), so they remain nonnegative
and we have $-{\frak v}_k(t)\leq {\frak u}_k(t)\leq {\frak v}_k(t)$ for any $t>0$.  Corollary \ref{cor:v} therefore gives
\begin{equation}\label{eqn:ub}
\mathbb{P}\left(
|{\frak u}^{(\nu)}_k(t)|<\frac{\varphi^{10}}{N} \frac{1}{\max((\hat k/N)^{1/3},t)}\ {\rm for\ all}\ k\in\llbracket 1,N\rrbracket \ {\rm and}\  t\in[0,1]
\right)
>1-N^{-D}
\end{equation}
for all $N>N_0(D)$. Note in particular that $N_0$ does not depend on $\nu\in[0,1]$. 
The above equation easily implies that for any fixed $\wt D$ and $p$, for large enough $N$ we have $\E(|{\frak u}^{(\nu)}_k(t)|^{2p})<((C\varphi^{10})/(Nt))^{2p}+N^{-\wt D}$, so that by H\"{o}lder's inequality  we have
$$
\E(|\la_k(t)-\mu_k(t)|^{2p})= \E\left(\left|\int_{0}^1{\frak u}_k^{(\nu)}(t)\rd\nu\right|^{2p}\right)
\leq
\int_0^1\E(|{\frak u}_k^{(\nu)}|^{2p})\leq \left(\frac{\varphi^{10}}{Nt}\right)^{2p}+N^{-\wt D}.
$$
By  choosing $p=\lfloor10/\e\rfloor$ and $\wt D=D+100p$,  Markov's inequality concludes the proof for fixed $k$ and $t$.

By a simple union bound
the same estimate holds for the event simultaneously over all $k$ for $N\geq N_0(D+1)$. For simultaneity over $t$, a standard argument based on discretization in time and Weyl's inequality
to bound increments in small intervals concludes the proof.
\end{proof}

\subsection{Proof of Theorem \ref{Gaussian}.}\ Let $F$ be a given smooth and bounded test function. We rely on \cite{Gus2005,Oro2010}
so that we only need to prove
\begin{equation}\label{eqn:aim}
\E_{H} F\left(X_k\right)
=
\E_{\rm GOE} F\left(X_k\right)+\oo(1)
\end{equation}
for any diverging $k\in\llbracket 1,N^{1-\e}\rrbracket$.
From Theorem \ref{thm:edgerelaxation}, for 
$t>(k/N)^{1/3}N^{\e/10}$, we have
$$
\E F\left(c\ \frac{\lambda_{k}(t)-\gamma_{k}}{(\log k)^{1/2}N^{-2/3}k^{-1/3}}\right)
=
\E F\left(c\ \frac{\mu_{k}(t)-\gamma_{k}}{(\log k)^{1/2}N^{-2/3}k^{-1/3}}\right)+\oo(1),
$$
so that (\ref{eqn:aim}) holds for
any Gaussian divisible ensemble of type $\wt H_t=e^{-t/2}\wt H_0+(1-e^{-t})^{1/2}\ U$, 
where $\wt H_0$  is any  initial generalized  Wigner matrix and 
$U$ is an independent standard  GOE matrix. 
We now construct a generalized  Wigner matrix $\wt H_0$ such that  
the first three moments of  $\wt H_t$ 
match exactly those of the  target  matrix $H$ and the differences between 
the fourth moments of the two ensembles are less than 
$N^{-c}$ for some positive $c$. This existence of such a initial random variable 
is given for example by \cite[Lemma 3.4]{EYYBernoulli}. By the following 
Proposition \ref {t2}, we have 
$$
\E_{\wt H_t} F\left(X_k\right)
=
\E_{H} F\left(X_k\right)+\oo(1).
$$
The previous two equations conclude our proof of (\ref{eqn:aim}), and therefore  Theorem \ref{Gaussian} (the proof in the multidimensional case is analogue).\\

The following proposition is a slight extension of the Green's function comparison theorem from \cite{ErdYauYin2012Univ},
(see for example \cite[theorem 5.2]{BouYau2017} for an analogue statement for eigenvectors).
Compared to \cite{ErdYauYin2012Univ}, we include the following minor modifications: (1) We state it for energies in the entire spectrum. 
(2) We allow the test function to be $N$-dependent.

Proposition \ref{t2} can be proved exactly as in \cite{ErdYauYin2012Univ}, so we do not repeat
it. Note that at the edge, the 4 moment matching can be replaced by 2 moments \cite{ErdYauYin2012}. For our applications, this improvement is not necessary.

\begin{proposition} \label{t2}
Let $H^{\f v}$ and $H^{\f w}$ be generalized Wigner ensembles satisfying (\ref{eqn:tail}). 
Assume that the first three moments of the entries ($h_{ij}=\sqrt{N}H_{ij}$) are the same, i.e.
$
\E^{\f v}(h_{ij}^k)=  \E^{\f w}(h_{ij}^k)
$
for all $1\leq i\leq j\leq N$ and $1\leq k\leq 3$.
Assume also that there exists $\xi> 0$ such that 
$$
	\Big |  \E^{\f v}(h_{ij}^4)- \E^{\f w}(h_{ij}^4)\Big | \le N^{-\xi} 
\for i \leq j.
$$
Then there is $\e > 0$ depending on $\xi$ such that for any integer $k$,  any choice of indices  $1\le j_1, \ldots, j_k \le N$
and smooth bounded $\Theta:\mathbb{R}^k\to\mathbb{R}$,
$$
 \left(\E^{\f v} - \E^{{\f w}}\right) \Theta \pB{ 
N^{2/3}(\hat j_i)^{1/3}\lambda_{j_1}, \ldots, N^{2/3}(\hat j_k)^{1/3}\lambda_{j_k}} =\OO(N^{-\e}\max_{0\leq m\leq 5}\|\Theta^{(m)}\|_\infty^{10}).
$$
\end{proposition}

\subsection{Average estimate in the bulk.}\ Proposition \ref{prop:estimate} gave bounds on $\wt f_t(z)$, useful for universality at the edge of the spectrum. The following estimate has a similar proof and justifies (\ref{eqn:ft}) in the bulk of the spectrum. Although not used in this paper, it is an important ingredient to study fluctuations of random determinants in \cite{BouMod2018}.

\begin{proposition}\label{prop:estimatebulk}
Let $\kappa>0$ be a fixed (small) constant. Then for any $D>0$ there exists $N_0(D,\kappa)$ such that for 
any $N\geq N_0$ we have
$$
\mathbb{P}\left(|f_t(z)-f_0(z_t)|\leq \frac{\varphi^{30}}{N\eta}\ {\rm for\ all}\ 0<t<1\ {\rm and}\ z=E+\ii \eta, \frac{\varphi^2}{N}<\eta<1,|E|<2-\kappa\right)>1-N^{-D}.
$$
\end{proposition}

\begin{proof}
We strictly follow the proof of Proposition \ref{prop:estimate}. Actually, the only differences are (i) the observable, now $f$ instead of $\wt f$ (but the equations are the same), (ii) simplifications, 
as we now know the a priori bound (\ref{eqn:ub}), and some estimates become simpler in the bulk of the spectrum.

More precisely, for any $1\leq \ell,m,p\leq N^{10}$, we define $t_\ell=\ell N^{-10}$ and $z^{(m,p)}=E_m+\ii\eta_p$ where $\int_{-{\infty}}^{E_m}\rd\rho=(m-1/2) N^{-10}$ and $\eta_p=\frac{\varphi^2}{N}+pN^{-10}$.
We also define
\begin{align}
\tau_{\ell,m,p}&=\notag
\inf\left\{0\leq s\leq t_\ell:   |f_s(z^{(m,p)}_{t_\ell-s})-f_0(z^{(m,p)}_{t_\ell})|>\frac{\varphi^{25}}{N\eta_p}\right\},\\
\tau_1&=
\inf\left\{0\leq t\leq 1\mid \exists k\in\llbracket 1,N\rrbracket :|{\frak u}^{(\nu)}_k(t)|>\frac{\varphi^{10}}{N} \frac{1}{\max((\hat k/N)^{1/3},t)}\right\},\label{eqn:t1}\\
\tau&=\min\{\tau_0,\tau_1,\tau_{\ell,m,p}:0\leq \ell,m,p\leq N^{10},|E_m|<2-\kappa\},\notag
\end{align}
where we remind the definition $\tau_0$ from (\ref{eqn:t0}).
By following the argument between (\ref{eqn:inter1}) and (\ref{eqn:inter3}), we just need to prove that 
for any $D>0$ there exists $\wt N_0(\kappa,D)$ such that for any $N\geq \wt N_0$, we have 
\begin{equation}\label{eqn:inter4}
\mathbb{P}(\tau=1)>1-N^{-D},
\end{equation}
with the convention $\inf\varnothing=1$.
Let $t=t_{\ell}$, $z=z^{(m,p)}$ where $0\leq \ell,m,p\leq N^{10},|E_m|<2-\kappa$, and $g_u(z)=f_{u}(z_{t-u})$. 
As in (\ref{eqn:gev}), we have
$$
\rd g_{u\wedge \tau}(z)=\e_{u}(z_{t-u})\rd({u\wedge \tau})-\frac{e^{-u/2}}{\sqrt{N}}\sqrt{\frac{2}{\beta}}\sum_{k=1}^N\frac{{\frak u}_k({u})}{(z_{t-u}-x_k(u))^2}\rd B_k({u\wedge \tau})
$$
where
$
\e_{u}(z)=(s_u(z)-m(z))\partial_z  f_u+\frac{1}{N}\left(\frac{2}{\beta}-1\right)(\partial_{zz} f_u).
$

The first error term can be bounded as in (\ref{eqn:estim1}) and (\ref{estim2}), with the  simplification that
now  $\kappa(z)\sim b(z)\sim 1$, so that the exact same calculation gives $\sup_{0\leq s\leq t}|\int_0^{s}\e_{u}(z_{t-u})\rd({u\wedge \tau})|\leq \frac{\varphi}{N\im z}$.

For the second error term, as in (\ref{eqn:error2}), we need to bound 
the quadratic variation.
This step is simpler than in the proof of Proposition \ref{prop:estimate}, because we now have some a priori bound on ${\frak u}_k(s)$ before time $\tau_1$. Moreover, as $z_0$ is close to the bulk, we do not need 
Lemma \ref{lem:A2} and directly obtain
$$
\int_0^{t}\frac{1}{N}\sum_k\frac{{\frak u}_k(u)^2}{|z_{t-u}-\gamma_k|^4}\rd (u\wedge \tau)\leq \frac{\varphi^{21}}{N^2\eta^2}.
$$
By the previous estimates and a union bound, for any $D>0$ there exists $N_0$ such that for $N>N_0$
$$
\mathbb{P}\left(
\sup_{0\leq \ell,m,p\leq N^{10}, |E_m|<2-\kappa,0\leq s\leq t_\ell}
|f_{s\wedge\tau}(z^{(m)}_{t_\ell-s\wedge\tau})-f_{0}(z^{(m)}_{t_\ell})|>\frac{\varphi^{25}}{N\eta_p}
\right)<N^{-D}.
$$
Together with Lemma \ref{rig} and (\ref{eqn:ub}), this implies (\ref{eqn:inter4}) and the result.
\end{proof}


\section{Relaxation from a maximum principle}

\subsection{Result.}\ 
The main result of this section is the following. Again,  remember that $\bla$ and $\bmu$ satisfy the same equation (\ref{eqn:dbm}), with respective initial conditions a generalized Wigner and GOE spectrum.
Remember the notation (\ref{eqn:zteq}); let $\gamma_k^t=(\gamma_k)_t$ (with the convention $\gamma^t=(\gamma+\ii0^+)_t$ when $z=\gamma\in\mathbb{R}$) and 
\begin{equation}\label{eqn:aver}
\bar {\frak u}_k(t)=\frac{1}{N\im m(\gamma_k^t)}\sum_{j=1}^N\im\left(\frac{1}{\gamma_j-\gamma_k^t}\right)(\la_j(0)-\mu_j(0)).
\end{equation}
The following theorem improves homogenization estimates which appeared first in \cite{BouErdYauYin2016}, both in terms of the scale and the probability bounds.

\begin{theorem} \label{thm:average}Consider the dynamics (\ref{eqn:dbm}) (or its Hermitian ensemble counterpart).
Let $\alpha,\varepsilon>0$ be fixed, arbitrarily small. For any  (large) $D>0$, there exist $C,N_0$ such that for any $N\geq N_0$, $\varphi^C/N<t<1$, and  $k\in\llbracket \alpha N,(1-\alpha)N\rrbracket$ we have
$$
\mathbb{P}\left(\ |(\la_{k}(t)-\mu_k(t)-\bar {\frak u}_k(t)|>\frac{N^{\varepsilon}}{N^2 t}\right)\leq N^{-D}.
$$
\end{theorem}

\begin{remark}\label{rem:anyBeta2}
The same comment as Remark \ref{rem:anyBeta} holds: The above  theorem is true for any $\beta\geq 1$ in (\ref{eqn:beta}), provided $\bla$ and $\bmu$ satisfy optimal initial rigidity estimates.
\end{remark}

\begin{corollary}\label{cor:gaps}
Let $\alpha,\varepsilon>0$ be fixed, arbitrarily small. Then for any (large) $D>0$, there exist $C,N_0$ such that for any $N>N_0$, $\varphi^C/N<t<1$ and $k\in\llbracket \alpha N,(1-\alpha)N\rrbracket$ we have 
$$
\mathbb{P}\left(|(\lambda_{k+1}(t)-\lambda_k(t))-(\mu_{k+1}(t)-\mu_k(t))|>\frac{N^{\varepsilon}}{N^2 t}\right)\leq N^{-D}.
$$
\end{corollary}

\begin{proof}
Note that
$$
|(\lambda_{k+1}(t)-\lambda_k(t))-(\mu_{k+1}(t)-\mu_k(t))|
\leq
|(\la_{k}(t)-\mu_{k}(t)-\bar {\frak u}_k(t)|
+
|(\la_{k+1}(t)-\mu_{k+1}(t)-\bar {\frak u}_{k+1}(t)|
+
|\bar {\frak u}_{k+1}(t)-\bar {\frak u}_{k}(t)|.
$$
From Theorem \ref{thm:average}, the first two terms do not exceed $\varphi^C/(N^2 t)$ with  probability $1-N^{-D}$.
The third term is bounded by the following Lemma \ref{lem:init}, an elementary consequence of rigidity and dynamics of the $\gamma$'s. This concludes the proof.
\end{proof}

\begin{lemma}\label{lem:init}
For any  $\alpha>0$, there exists a constant $C>0$ such that for any $(k,\ell)\in\llbracket\alpha N,(1-\alpha)N\rrbracket^2$, $|E|<2-\alpha$ and $s,t,\eta\in[\varphi^2/N,1]$,
in the set $\mathscr{A}$ from (\ref{eqn:A}) we have (here $z=E+\ii\eta$)
\begin{align}
&|\bar{\frak u}_k(t)-\bar{\frak u}_\ell(s)|\leq C\varphi \left(\frac{|k-\ell|}{N^2\min(s,t)}+\frac{|t-s|}{N\min(s, t)}\right)\label{elem1},\\
&\left|\frac{\im f_0(z_t)}{N\im s_0(z_t)}-\bar{\frak u}_\ell(s)\right|\leq C\varphi \left(\frac{|E-\gamma_\ell|}{N\min(s,\eta+t)}+\frac{|\eta+t-s|}{N\min(s,\eta+t)}\right)\label{elem2}.
\end{align}
\end{lemma}

\begin{proof}
As preliminary elementary estimates, there exists a constant $C>0$ such that in the required range of $k,\ell,s,t$ we have
\begin{equation}\label{ele1}
\im m(\gamma_k^t),\im m(\gamma_\ell^s)>C^{-1},\ |\im m(\gamma_k^t)-\im m(\gamma_\ell^s)|+|\gamma_k^t-\gamma_\ell^s|<C\left(
\frac{|k-\ell|}{N}+|s-t|
\right).
\end{equation}
We detail the proof of the first inequality above. From (\ref{eqn:zteq}) there exists  a compact set $\mathscr{C}=\mathscr{C}(\alpha)$ which does not depend on $N$ and does not intersect $(-\infty,2]\cup[2,\infty)$ such that for any 
$k\in\llbracket\alpha N,(1-\alpha)N\rrbracket$ and $0<t<1$, $\gamma_k^t\in \mathscr{C}$. The required inequality then follows from $\inf_{z\in\mathscr{C}}\im m(z)>0$. The second inequality of (\ref{ele1}) follows from the same argument together with the observation that $m$ is Lipschitz on $\mathscr{C}$.

 Moreover, in $\mathscr{A}$ the rigidity estimates gives 
 $|\la_j(0)-\mu_j(0)|\leq C\varphi^{1/2}N^{-\frac{2}{3}}(\hat j)^{-\frac{1}{3}}$, so that the same proof as Lemma \ref{f0} gives
\begin{equation}\label{ele2}
\left|\frac{1}{N}\im\sum_j\frac{\la_j(0)-\mu_j(0)}{\gamma_j-\gamma_k^t}\right|
\leq C\frac{\varphi^{1/2}}{N}.
\end{equation}
We decompose
\begin{multline*}
|\bar{\frak u}_k(t)-\bar{\frak u}_\ell(s)|\leq \frac{1}{N}
\left|\left(\frac{1}{\im m(\gamma_k^t)}-\frac{1}{\im m(\gamma_\ell^s)}\right)\sum_{j=1}^N\im\left(\frac{\la_j(0)-\mu_j(0)}{\gamma_j-\gamma_k^t}\right)\right|\\
+
\frac{1}{N\im m(\gamma_\ell^s)}\sum_{j=1}^N\left|\im\left(\frac{1}{\gamma_j-\gamma_k^t}\right)-
\im\left(\frac{1}{\gamma_j-\gamma_\ell^s}\right)
\right||\la_j(0)-\mu_j(0))|.
\end{multline*}
From (\ref{ele1}) and (\ref{ele2}), the first line is at most $\varphi^{1/2}\left(
\frac{|k-\ell|}{N^2}+\frac{|s-t|}{N}
\right)$, while the second is bounded in $\mathscr{A}$ by
$$
\frac{C\varphi^{1/2}}{N}\sum_j N^{-\frac{2}{3}}(\hat j)^{-\frac{1}{3}} \left|\frac{\gamma_k^t-\gamma_\ell^s}{(\gamma_j-\gamma_k^t)(\gamma_j-\gamma_\ell^s)}\right|
\leq
C\varphi^{1/2}\left(\frac{|k-\ell|}{N^2}+\frac{|s-t|}{N}\right)\left(
\sum_{j}\frac{N^{-\frac{2}{3}}(\hat j)^{-\frac{1}{3}}}{|\gamma_j-\gamma_k^t|^2}
+
\sum_{j}\frac{N^{-\frac{2}{3}}(\hat j)^{-\frac{1}{3}}}{|\gamma_j-\gamma_\ell^s|^2}
\right).
$$
As $\im \gamma_k^t\sim t, \im \gamma_\ell^s\sim s$, each sum above is $\OO(\min(s,t)^{-1})$. This concludes the proof of (\ref{elem1}). The proof of (\ref{elem2}) is the same.
\end{proof}

\subsection{Proof of Theorem \ref{thm:average} by induction.}\ 
We implement an iterative scheme to reach the optimal error term. Some inspiration from this scheme comes from \cite[Section 3]{PartI}, although the induction there quantifies eigenvectors
delocalization instead of eigenvalues, and many aspects of the proof are different.
Consider the
following property, for a parameter $0<a\leq 1$.\\

\noindent {\bf Property {\boldmath (${\rm P}_a$)}}.
For any fixed (small) $\alpha>0$ and (large) $D>0$, 
 there exist $C$ and $N_0$ such that for any $\nu\in[0,1]$, the following holds with probability at least $1-N^{-D}$. For any $\varphi^C/N<t<1$, $k\in\llbracket \alpha N,(1-\alpha)N\rrbracket$ and $N\geq N_0$,
\begin{equation}\label{Pa}
|{\frak u}^{(\nu)}_k(t)-\bar {\frak u}_k(t)|<\varphi^C\frac{(Nt)^a}{N^2t}.
\end{equation}
Theorem \ref{thm:average} is a consequence of the following two propositions. 

\begin{proposition}\label{prop:init}
$({\rm P}_1)$ holds.
\end{proposition}
\begin{proof}
From (\ref{eqn:ub}), we know that ${\frak u}_k(t)=\OO(\varphi^C N^{-1})$ with overwhelming probability, uniformly in the required range of parameters. We also have $\bar {\frak u}_k(t)=\OO(\varphi^CN^{-1})$ thanks to 
the definition (\ref{eqn:aver}) and the rigidity estimate
$
|{\frak u}_j(0)|<\varphi^{1/2}N^{-2/3}\hat{j}^{-1/3}
$
(see Lemma \ref{rig}). This concludes the proof.
\end{proof}

\begin{proposition}\label{prop:keyrec}
If $({\rm P}_a)$ holds, so does $({\rm P}_{3a/4})$.
\end{proposition}

\begin{proof}[Proof of Theorem \ref{thm:average}] Let $\e>0$. By initialization with Proposition \ref{prop:init} and a finite number of iterations of Proposition \ref{prop:keyrec}, 
for any fixed (small) $\alpha>0$ and (large) $D>0$, 
there exist $C$ and $N_0$ such that for any $\nu\in[0,1]$, $N\geq N_0$,
$$
\mathbb{P}
\left(
|{\frak u}^{(\nu)}_k(t)-\bar {\frak u}_k(t)|<\frac{N^\e}{N^2 t}
\ {\rm for\ all}\ k\in\llbracket \alpha N,(1-\alpha)N\rrbracket \ {\rm and}\  t\in[\varphi^C/N,1]
\right)
>
1-N^{-D}.
$$
The same estimate holds after integration over $\nu\in[0,1]$, with rigorous justification given by large moments and Markov's inequality, similarly to the argument  after (\ref{eqn:ub}).
\end{proof}

\noindent The remaining part of this section proves Proposition \ref{prop:keyrec}. It relies on the following three lemmas.

The first lemma is an approximation of our dynamics (\ref{eqn:gene}) with short range dynamics. Such approximations 
for the analysis of the Dyson Brownian motion appeared first in \cite{ErdYau2015}. Our version assumes property $({\rm P}_a)$ and gives a better bound.
Remember we defined $c_{jk}=c_{jk}(s)=1/(N(x_j(s)-x_k(s))^2)$ and  write $\mathscr{B}=\mathscr{S} + \mathscr{L}$,
\begin{align}
&(\mathscr{S} f)(k)  =   \sum_{|j-k| \le \ell }   c_{jk}(s) \left(f(j) - f(k)\right),\label{eqn:shortcut}\\
&(\mathscr{L}f) (k) =    \sum_{|j-k| >  \ell }   c_{jk}(s) \left(f(j) - f(k)\right),\notag
\end{align}
for some parameter $\ell=\ell(N,a)$  chosen later. 
Denote by $ \rU_{\mathscr{S}} ( s,   t)$  
 the  semigroup associated with $\mathscr{S}$ from time $s$ to time $t$,  i.e.
$
\partial_{t} \rU_{\mathscr{S}} (s,t) = \mathscr{S}(t)  \rU_{\mathscr{S}} (s,t)
$
 and $\rU_{\mathscr{S}} (s,s)=\Id$.
The notation $\rU_{\mathscr{B}}(s,t)$ is analogous.

\begin{lemma}[Short range approximation]\label{lem:shortRange}
Assume $({\rm P}_a)$.
For any fixed (small) $\alpha>0$ and (large) $D>0$, 
 there exist $C,N_0$ (depending on $\alpha,a,D$) such that
 the following holds with probability at least $1-N^{-D}$.
For any $N>N_0$, $\varphi^C/N<t<1$, $u<v$ in $[t/2,t]$, $\ell>\varphi$ and $k\in\llbracket\alpha N,(1-\alpha)N\rrbracket$, 
\begin{equation}\label{eqn:shortRange}
\left|\left(\left( \rU_{\mathscr{B}} (u,v) - \rU_{\mathscr{S}} (u,v)\right)\frak{u}(u)\right)(k)\right|\leq \varphi^C
|u-v|
\left(
\frac{N}{\ell}\frac{(Nt)^a}{N^2t}
+
\frac{1}{Nt}
\right).
\end{equation}
\end{lemma}

The second lemma is a finite speed of propagation  for the dynamics defined by (\ref{eqn:shortcut}). Such estimates  appeared first in \cite{ErdYau2015}, here we state the version from \cite[Lemma 6.2]{BouYau2017}, optimal in terms of distance and probability bound. The version below is simpler than \cite[Lemma 6.2]{BouYau2017} as it corresponds to the one-particle case, and we change the condition $|i-j|>N^\e\ell $ into $|i-j|>\varphi\ell $ for convenience, the proof being unchanged.

\begin{lemma}[Finite speed of propagation]\label{lem:short}
For any fixed (small) $\alpha>0$ and (large) $D>0$, 
 there exists $N_0$ (depending on $\alpha,D$) such that the following holds with probability at least $1-N^{-D}$.  For any $N>N_0$,
$0<u<v<1$, $\ell\geq N|u-v|$,
$k\in\llbracket\alpha N,(1-\alpha)N\rrbracket$ and $j\in\llbracket1,N\rrbracket$ such that $|k-j|>\varphi \ell$, we have
\begin{equation}\label{eqn:short}
(\rU_{\mathscr{S}} (u,v)\delta_k)(j)<N^{-D}.
\end{equation}
\end{lemma}

For the third lemma, we consider (\ref{eqn:shortcut}) with a well-chosen initial condition, similarly to
\cite[Section 7.2]{BouYau2017}.
We fix some initial and final times $u<t$, the short range dynamics parameter $\ell$, the space window scale $r$ and always assume
\begin{equation}\label{eqn:parameters}
\varphi^{30}|u-t|<\varphi^{20}\frac{\ell}{N}<\varphi^{10}r<t.
\end{equation}
We also consider a fixed index $k$. Given this, we define 
$$
({\rm Flat}_b h)(j)=
\left\{
\begin{array}{l}
 h(j)\ {\rm if}\ |j-k|\leq b\\
 \bar{\frak u}_k(t)\ {\rm if}\ |j-k|>b\\
\end{array}
\right.,\ \ \
({\rm Av}\, h)(j)=\frac{1}{|\llbracket Nr,2Nr\rrbracket |}\sum_{Nr\leq b\leq 2Nr}({\rm Flat}_b h)(j).
$$
This averaging operator can also be written as a linear combination in terms of a Lipschitz function $a$:
\begin{equation}\label{eqn:coefprime}
({\rm Av} h)(j)=a_j h(j)+(1-a_j)\bar{\frak u}_k(t)\ \ 
{\rm where}\  |a_i-a_j|\leq \frac{|i-j|}{Nr}.
\end{equation}
The function ${\rm Av}$, introduced in \cite{BouYau2017}, allows to flatten the initial condition outside a large box, and keep the actual observable in a smaller box. For the purpose of further estimates, this interpolation with a constant at $\infty$ needs to be regular enough; a linear interpolation in the window 
$\llbracket Nr, 2Nr\rrbracket$ is sufficient for our purpose.

Finally, let ${\frak w}$ be the solution of
\begin{align*}
\frac{\rd}{\rd s}\frak{w}_j(s)&=(\mathscr{S}(s){\frak w})(j),\ u<s<t,\\ 
\frak{w}(u)&={\rm Av}\, \frak{u}(u).
\end{align*}
The following lemma provides good estimates on averages of the $\frak{w}_j$'s. The stochastic advection equation satisfied by $f_t$ will be essential for its proof (see Lemma \ref{lem:improved}).

\begin{lemma}[Average of the modified dynamics]\label{lem:averageMod} Assume $({\rm P}_a)$.
For any fixed (small) $\alpha>0$ and (large) $D>0$, 
 there exist $N_0,C$ (depending on $\alpha,a,D$) such that
the following holds with probability at least $1-N^{-D}$.
For any $N>N_0$,  $\varphi^C/N<\eta,t<1$, 
$u<s$ in $[t/2,t]$, $\ell>\varphi$, $j,k\in\llbracket \alpha N,(1-\alpha) N\rrbracket$ such as $|\gamma_j-\gamma_k|<10r$, $z=\gamma_j+\ii\eta$, we have (remember ${\frak w}$ depends on $k$ and $u$)
\begin{multline}\label{eqn:averageMod}
\frac{1}{N}\im \sum_{|i-j|<\ell}\frac{{\frak w}_i(s)}{x_i(s)-z}-
\left(\frac{1}{N}\im \sum_{|i-j|<\ell}\frac{1}{x_i(s)-z}\right)\bar{\frak u}_k(s)\\
=\OO(\varphi^C)\left(
\frac{r}{Nt}+\frac{\eta}{Nt}+\frac{(Nt)^a}{N^2t}\left(\frac{\ell}{Nr}+\frac{N\eta}{\ell}+\frac{N|u-t|}{\ell}+\frac{1}{N\eta}\right)
\right).
\end{multline}
\end{lemma}

Based on the previous lemmas, we can now complete the proof of Proposition \ref{prop:keyrec}.
Until the end of this proof, we fix $\alpha,D>0$ and find $N_0$ such that the conclusion of the three lemmas above hold with probability at least $1-N^{-D}$ for $N>N_0$, together with the rigidity estimate from Lemma \ref{rig}. We work on this good event, i.e. we assume that we are in $\mathscr{A}$ from (\ref{eqn:A}), and that (\ref{eqn:shortRange}), (\ref{eqn:short}) and (\ref{eqn:averageMod}) hold.

We fix some index $k\in\llbracket 2\alpha N,(1-2\alpha)N\rrbracket$. We have
$$
|{\frak u}_k(t)-{\frak w}_k(t)|\leq 
|\left(\left( \rU_{\mathscr{B}} (u,t) - \rU_{\mathscr{S}} (u,t)\right)\frak{u}(u)\right)(k)|
+
|\left( \rU_{\mathscr{S}} (u,t)(\frak{u}(u)-\frak{w}(u))\right)(k)|.
$$
We can bound the first term on right-hand side with  (\ref{eqn:shortRange}). Moreover,
note that $\frak{u}(u)-\frak{w}(u)$ is supported on $\{j:|j-k|>Nr\}$ because of $\frak{w}(u)={\rm Av}\, \frak{u}(u)$ and the averaging operator does not change functions in $\{j:|j-k|\leq Nr\}$. Hence by (\ref{eqn:short}) and the choice of parameters (\ref{eqn:parameters}) the second term above 
is $\OO(N^{-100})$.
We therefore obtained
\begin{equation}\label{eqn:uks}
|{\frak u}_k(t)-{\frak w}_k(t)|\leq 
\varphi^C|u-t|
\left(
\frac{N}{\ell}\frac{(Nt)^a}{N^2t}
+
\frac{1}{Nt}
\right)
+N^{-100}.
\end{equation}
We now evaluate ${\frak{w}}_k(t)$, by considering two cases.

Assume first that there exist an index $j$ and a time $s\in[u,t]$ such that
${\frak{w}}_j(s)-\bar {\frak{u}}_k(t)=
M(s):=\max_{1\leq i\leq N}({\frak{w}}_i(s)-\bar{\frak{u}}_k(t))$ and $|j-k|>3Nr$. 
As $i\mapsto{\frak w}_i(u)-\bar{\frak u}_k(t)$ is compactly supported on $\{i:|i-k|<2Nr\}$, by the finite speed of propagation estimate  Lemma \ref{lem:short} we have ${\frak w}_j(s)-\bar{\frak u}_k(t)\leq N^{-100}$. By the parabolic maximum principle,  $M$ decreases, which implies 
\begin{equation}\label{eqn:trivialcase}
M(t)<N^{-100}.
\end{equation}

Secondly, assume that for any $s$, for all $j$ such that ${\frak{w}}_j(s)-{\frak{u}}_k(t)=
M(s)$ we have $|j-k|<3Nr$.
For any such $j$ and $\eta>0$, we have
\begin{multline*}
\frac{\rd}{\rd s}({\frak{w}}_j(s)-\bar{\frak{u}}_k(t)))
=
\sum_{|i-j|\leq\ell} \frac{{\frak{w}}_i(s)-{\frak{w}}_j(s)}{N(x_j(s)-x_i(s))^2}
\leq\frac{1}{N}
\sum_{|i-j|\leq\ell} \frac{{\frak{w}}_i(s)-{\frak{w}}_j(s)}{(x_j(s)-x_i(s))^2+\eta^2}\\
=
\frac{1}{N\eta}\left(\im\sum_{|i-j|\leq \ell}\frac{{\frak{w}}_i(s)}{x_i(s)-z_j}\right)
-\frac{1}{N\eta}\left(\im\sum_{|i-j|\leq \ell}\frac{1}{x_i(s)-z_j}\right){\frak{w}}_j(s)
\end{multline*}
where $z_j=x_j(s)+\ii\eta$. By Lemma \ref{lem:averageMod} and the observation $|\bar {\frak u}_k(s)-\bar {\frak u}_k(t)|<C\varphi \frac{|u-t|}{Nt}$ from Lemma \ref{lem:init}, the first parenthesis above can be evaluated so that, if we denote $f'(x^+)$ the right derivative\footnote{$M$ is the maximum of $N$ smooth curves, so its right derivative exists and is bounded by the max of all individual derivatives where the maximum occurs.} of a function $f$ at $x$, we have
$$
\frac{\rd}{\rd s}M(s^+)\leq 
-\frac{c}{\eta} M(s)+\frac{\varphi^C}{\eta}\left(
\frac{r}{Nt}+\frac{\eta}{Nt}+\frac{(Nt)^a}{N^2t}\left(\frac{\ell}{Nr}+\frac{N\eta}{\ell}+\frac{N|u-t|}{\ell}+\frac{1}{N\eta}\right)\right).
$$
Note that the error term due to $|\bar {\frak u}_k(s)-\bar {\frak u}_k(t)|<C\varphi \frac{|u-t|}{Nt}$ has been absorbed above in $\frac{r}{Nt}$ because $|u-t|\ll r$.
If we choose $\eta=|t-u|/\varphi$, the above equation implies
$$
M(t)\leq \varphi^C\left(
\frac{r}{Nt}+\frac{(Nt)^a}{N^2t}\left(\frac{\ell}{Nr}+\frac{N|t-u|}{\ell}+\frac{1}{N|t-u|}\right)
\right).
$$
With the optimal choice
\begin{equation}\label{eqn:para}
r=\frac{(Nt)^{\frac{3a}{4}}}{N},\ \ell=(Nt)^{\frac{a}{2}},\ |u-t|=\frac{(Nt)^{\frac{a}{4}}}{N},
\end{equation}
we obtain
$
M(t)\leq \varphi^{C}(Nt)^{\frac{3a}{4}}/(N^2t).
$
This inequality is also true in the case (\ref{eqn:trivialcase}).

By the same reasoning we obtain the same bound holds for $-\min_{1\leq i\leq N}({\frak{w}}_i(t)-{\frak{u}}_k(t))$, so that in particular
$$|{\frak{w}}_k(t)-\bar{\frak{u}}_k(t)|\leq  \varphi^{C}\frac{(Nt)^{\frac{3a}{4}}}{N^2t}.$$
Together with the estimate (\ref{eqn:uks}) with parameters (\ref{eqn:para}), this shows that
$
|{\frak{u}}_k(t)-\bar{\frak{u}}_k(t)|\leq  \varphi^{C}\frac{(Nt)^{\frac{3a}{4}}}{N^2t}
$
for all index $k\in\llbracket 2\alpha N,(1-2\alpha)N\rrbracket$. As $\alpha$ is arbitrary, this concludes the proof of Proposition \ref{prop:keyrec}.

\subsection{Proof of Lemma \ref{lem:shortRange}}.
We fix $\alpha,D>0$ and find $N_0$ such that the conclusion of Lemma \ref{rig},  Lemma \ref{lem:short} and Property {(${\rm P}_a$)} hold for $N>N_0$, with probability at least $1-N^{-D}$. We work on this good event, i.e. we assume that we are in $\mathscr{A}$ from (\ref{eqn:A}), and that (\ref{Pa}) and (\ref{eqn:short}) hold.

By Duhamel's formula, we have
$$
\left(\left( \rU_{\mathscr{B}} (u,v) - \rU_{\mathscr{S}} (u,v)\right)\frak{u}(u)\right)(k)
=
\int_u^v \left(\rU_{\mathscr{S}} (s,v)\mathscr{L}(s){\frak u}(s)\right)(k)\rd s.
$$
By the finite speed of propagation (\ref{eqn:short}), for any $k\in\llbracket 3\alpha N,(1-3\alpha)N\rrbracket$ we have 
$$
\left(\rU_{\mathscr{S}} (s,v)\mathscr{L}(s){\frak u}(s)\right)(k)=
\left(\rU_{\mathscr{S}} (s,v)(\mathscr{L}(s){\frak u}(s))\mathds{1}_{\llbracket 2\alpha N,(1-2\alpha)N\rrbracket}\right)(k)
+\OO(N^{-D}).
$$
The above equations together with $\rU_{\mathscr{S}}$ being a contraction for $L^\infty$, this implies that
$$
\left|\left(\left( \rU_{\mathscr{B}} (u,v) - \rU_{\mathscr{S}} (u,v)\right)\frak{u}(u)\right)(k)\right|
\leq
|u-v|\sup_{j\in\llbracket 2\alpha N,(1-2\alpha)N\rrbracket,u<s<v}\left|\mathscr{L}(s){\frak u}(s)(j)\right|+\OO(N^{-D}).
$$
Finally, from (\ref{Pa}) and Lemma \ref{lem:init}, for any $s$ in $[t/2,t]$, for any $i\in\llbracket \alpha N,(1-\alpha)N\rrbracket$ we have 
$|\frak{u}_i(s)-\frak{u}_j(s)|\leq 
\varphi^C(
\frac{(Nt)^a}{N^2t}
+
\frac{|i-j|}{N^2 t})
$,
and for $i\not\in\llbracket \alpha N,(1-\alpha)N\rrbracket$, (\ref{eqn:ub}) implies $|\frak{u}_i(s)-\frak{u}_j(s)|\leq \varphi^C N^{-2/3}(\hat i)^{-1/3}$.
This implies 
\begin{multline*}
\mathscr{L}(s){\frak u}(s)(j)=
\sum_{|i-j|>\ell}\frac{\frak{u}_i(s)-\frak{u}_j(s)}{N(x_i-x_j)^2}
=\OO\left(N\varphi^C\right)
\sum_{|i-j|\geq \ell}\frac{\frac{(Nt)^a}{N^2t}
+
\frac{|i-j|}{N^2 t}}{(i-j)^2}
+\OO\left(\varphi^C/N\right)\sum_{1\leq i\leq N}N^{-2/3}(\hat i)^{-1/3}
\\
=
\OO(\varphi^C)
\left(
\frac{N}{\ell}\frac{(Nt)^a}{N^2t}
+
\frac{1}{Nt}
\right),
\end{multline*}
where we also used $|x_i(s)-x_j(s)|>c|i-j|/N$, by rigidity together with $\ell>\varphi$. We therefore obtained (\ref{eqn:shortRange}) for $k\in\llbracket 3\alpha N,(1-3\alpha)N\rrbracket$. As $\alpha$ is arbitrary, this concludes the proof.

\subsection{Proof of Lemma \ref{lem:averageMod}.}\ 
We start with the following key improvement on local averages. Remember the notations (\ref{eqn:observable}) and (\ref{eqn:st}).

\begin{lemma}[Improved estimate on the local average]\label{lem:improved}
Assume $({\rm P}_a)$.
For any fixed (small) $\kappa>0$ and (large) $D>0$, 
 there exist $C$ and $N_0$ (depending on $a,\alpha,D$) such that
 the following holds with probability at least $1-N^{-D}$.
 For any  $t$ and $z=E+\ii\eta$,
satisfying $0<t<1$, $\varphi^CN^{-1}<\eta<1$, $|E|<2-\kappa$, we have
\begin{equation}\label{eqn:lemres}
\left| \im f_t(z)-e^{-t/2}\frac{\im s_{t}(z)}{\im s_0(z_t)}\im  f_0(z_t)\right|\leq \varphi^C\left(\frac{(Nt)^a}{N^2t\eta}
+\frac{1}{Nt}
\right).
\end{equation}
\end{lemma}

\noindent Note that for the initial iteration of Proposition \ref{prop:keyrec}, we have $a=1$ so the above estimate was already proved: an upper bound $\varphi^C/(N\eta)$ is  known by Proposition \ref{prop:estimatebulk}.
Hence, the above lemma is not necessary to obtain $({\rm P}_{3/4})$ and therefore relaxation of the Dyson Brownian motion. We only use it for optimal error bounds.

\begin{proof} For fixed $t$, consider the function
$$
h_{u}(z)=h_{u}^{(t)}(z)=f_u(z_{t-u})-\frac{\im f_0(z_t)}{\im s_0(z_t)}e^{-u/2}s_{u}(z_{t-u}),\ 0\leq u\leq t.
$$
Note that both $f$ and $e^{-u/2}s_u$ satisfies the stochastic advection equation (\ref{eqn:dynamicsPDE}), with  ${\frak u}_k(u)$ replaced by $1/N$ in the simpler case of $s$. By linearity, this implies that $h$ satisfies  the equation 
\begin{multline}
\rd h_u=
(s_u(z_{t-u})-m(z_{t-u}))\left(\partial_z f_u(z_{t-u})-\frac{\im f_0(z_t)}{\im s_0(z_t)}e^{-u/2}\partial_z s_u(z_{t-u})\right)\rd u\\
+\frac{1}{N}\left(\frac{2}{\beta}-1\right)\left(\partial_{zz}h_u(z_{t-u})-\frac{\im f_0(z_t)}{\im s_0(z_t)}e^{-u/2}\partial_{zz} s_u(z_{t-u})\right)\rd u-\frac{e^{-u/2}}{\sqrt{N}}\sqrt{\frac{2}{\beta}}\sum_{q=1}^N\frac{{\frak r}_q({u})\rd B_q(u)}{(z_{t-u}-x_q(u))^2},\label{eqn:heq}
\end{multline}
where
$
\frak{r}_q(u)={\frak u}_q(u)-\frac{\im f_0(z_t)}{N\im s_0(z_t)}.
$
We will use this equation to bound $\im (h_t-h_0)$ (i.e. the left-hand side of (\ref{eqn:lemres})) in a way similar to the proof of Proposition \ref{prop:estimatebulk}, with the novelty that our estimate on $\frak{r}_q(u)$ depends 
on the hypothesis (${\rm P}_a$) and  improves with small $a$.

As in the proof of Proposition \ref{prop:estimatebulk}, we define   $t_\ell=\ell N^{-10}$ and $z^{(m,p)}=E_m+\ii\eta_p$ where $\int_{-{\infty}}^{E_m}\rd\rho=(m-1/2) N^{-10}$ and $\eta_p=\frac{\varphi^2}{N}+pN^{-10}$. We pick $\alpha$ such that $\lfloor \alpha N\rfloor={\rm argmin}_q|\gamma_q-(-2+\kappa/10)|$.
Let 
\begin{align*}
\tau_{\ell,m,p}&=
\inf\left\{0\leq u\leq t_\ell:   |\im h_u^{(t_\ell)}(z^{(m,p)})|>\varphi^C\left(\frac{(Nt_\ell)^a}{N^2t_\ell\eta_p}+\frac{1}{Nt_\ell}\right)\right\},\\
\tau_2&=
\inf\left\{\frac{\varphi^C}{N}\leq u\leq 1\mid \exists q\in\llbracket \alpha N,(1-\alpha)N\rrbracket :|{\frak u}_q(u)-\bar {\frak u}_q(u)|>\varphi^C\frac{(Nu)^a}{N^2 u}\right\},\\
\tau&=\min\{\tau_0,\tau_1,\tau_2,\tau_{\ell,m,p}:0\leq \ell,m,p\leq N^{10},|E_m|<2-\kappa\},
\end{align*}
where $\tau_0,\tau_1$ are defined in (\ref{eqn:t0}) and (\ref{eqn:t1}), and our convention is $\inf\varnothing=1$.
By the same argument as  in between (\ref{eqn:inter1}) and (\ref{eqn:inter3}), we just need to prove that 
for any $D>0$ there exists $\wt N_0(\kappa,D)$ such that for any $N\geq \wt N_0$, we have 
\begin{equation}\label{eqn:inter44}
\mathbb{P}(\tau=1)>1-N^{-D}.
\end{equation}
Let $t=t_{\ell}$, $z=z^{(m,p)}=E+\ii\eta$ where $0\leq \ell,p\leq N^{10},|E_m|<2-\kappa$. 
We  now divide the proof in two steps.\\

\noindent{\it First step: a priori estimate on ${\frak r}_q$}. We claim that for any $\alpha>0$ there exists $C>0$ such that for any $\frac{\varphi^C}{N}\leq u\leq \tau$ and $q\in\llbracket \alpha N,(1-\alpha)N\rrbracket$  we have (remember that ${\frak r}_q(u)$ depends on $z$ and $t$)
\begin{equation}\label{eqn:rkbound}
|{\frak r}_q(u)|\leq \varphi^C\left(\frac{|\gamma_q-E|+|u-t|}{N u}+\frac{(Nu)^a}{N^2u}+\frac{\eta}{Nt}\right)=:
\varphi^C\left(\frac{|\gamma_q-E|}{Nu}+g(a,N,\eta,u,t)\right).
\end{equation}
For the proof, we choose $j$ such that $|\gamma_j-E|\leq \varphi/N$ and  write 
 $
 |{\frak r}_q(u)|\leq 
 |{\frak u}_q(u)-\bar{\frak u}_q(u)|
+
 |\bar{\frak u}_q(u)-\bar{\frak u}_j(t)| 
+
\left|\bar{\frak u}_j(t)-\frac{\im f_0(z_t)}{N\im s_0(z_t)}\right|.
$
We use $u\leq \tau_2$ to bound the first term,  (\ref{elem1}) for the second and (\ref{elem2}) for the third. This gives (\ref{eqn:rkbound}).

Note that we also have the more elementary estimate (useful for small $u$ or $q$ close to the edge)
\begin{equation}\label{eqn:rkbound2}
 |{\frak r}_q(u)|\leq C\varphi^{10}N^{-\frac{2}{3}}(\hat q)^{-\frac{1}{3}}.
\end{equation}
This is obtained by combining two estimates. First, we have $u\leq \tau_1$ so that $|{\frak u}_q(u)|\leq \varphi^{10}N^{-\frac{2}{3}}(\hat q)^{-\frac{1}{3}}$. Second, uniformly in  $E$ is in the bulk of the spectrum and $t<1$ we have
$\im s_0(z_t)>c>0$, which together with Lemma \ref{f0}  gives  $\frac{\im f_0(z_t)}{N\im s_0(z_t)}\leq C\varphi^{1/2}/N\leq C\varphi^{10}N^{-\frac{2}{3}}(\hat q)^{-\frac{1}{3}}$.\\

\noindent{\it Second step: bound on the increments}. The error term for $\sup_{0\leq s\leq t}|\im (h_s-h_0)|$ corresponding to the first line of (\ref{eqn:heq}) above can be bounded similarly to (\ref{eqn:estim1}), giving
\begin{equation}\label{eqn:intt1}
\int_0^{t\wedge\tau}|s_u(z_{t-u})-m(z_{t-u})|\left|\partial_z f_u(z_{t-u})-\frac{\im f_0(z_t)}{\im s_0(z_t)}\partial_z s_u(z_{t-u})\right|\rd u
\leq 
\int_0^{t\wedge \tau}\frac{\varphi}{N\im(z_{t-u})}\sum_{q=1}^N\frac{|{\frak r}_q(u)|}{|z_{t-u}-\gamma_q|^2}\rd u.
\end{equation}
In the above right-hand side, the terms  $\hat q\leq\alpha N$ are bounded with
(\ref{eqn:rkbound2}) and give a contribution
$$
C \int_0^{t\wedge \tau}\frac{\varphi}{N\im(z_{t-u})}\sum_{q=1}^N
\varphi^{10}N^{-\frac{2}{3}}(\hat q)^{-\frac{1}{3}}\rd u
\leq 
\int_0^{t\wedge \tau}\frac{\varphi}{N\im(z_{t-u})}\rd u\leq\frac{\varphi\log N}{N}.
$$
For the contribution from the bulk indices in the right-hand side of (\ref{eqn:intt1}), for $u\leq \tau$ we have (we abbreviate $g$ for $g(a,N,\eta,u,t)$)
$$
\sum_{q=1}^N\frac{|{\frak r}_q(u)|}{|z_{t-u}-\gamma_q|^2}\leq
\varphi^Cg\sum_{q=1}^N\frac{1}{|z_{t-u}-\gamma_q|^2}
+\frac{\varphi^C}{Nu}\sum_{q=1}^N\frac{|\gamma_q-E|}{|z_{t-u}-\gamma_q|^2}
\leq 
\frac{\varphi^Cg N}{\eta+(t-u)}
+\frac{\varphi^C}{u},
$$
so that
\begin{multline*}
\int_{\varphi^C/N}^{t\wedge \tau}\frac{\varphi}{N\im(z_{t-u})}\sum_{\hat q>\alpha N}^N\frac{|{\frak r}_q(u)|}{|z_{t-u}-\gamma_q|^2}\rd u\leq
\int_{\varphi^C/N}^{t}\frac{\varphi^C\rd u}{\eta+t-u}\left(\left(\frac{(Nu)^a}{N^2u}+\frac{\eta}{Nt}+\frac{t-u}{Nu}\right)\frac{1}{\eta+(t-u)}
+\frac{1}{Nu}\right)\\
\leq \varphi^C\left(\frac{(Nt)^a}{N^2t\eta}+\frac{1}{Nt}\right).
\end{multline*}
Finally, with (\ref{eqn:rkbound2}),
\begin{multline*}
\int_0^{\varphi^C/N}\frac{\varphi}{N\im(z_{t-u})}\sum_{\hat q>\alpha N}\frac{|{\frak r}_q(u)|}{|z_{t-u}-\gamma_q|^2}\rd u\leq
\int_0^{\varphi^C/N}\frac{\varphi^C}{N^2\im(z_{t-u})}\sum_{q=1}^N\frac{1}{|z_{t-u}-\gamma_q|^2}\rd u\\
\leq 
\int_0^{\varphi^C/N}\frac{\varphi^C}{N(\im(z_{t-u}))^2}\rd u\leq \frac{\varphi^C}{N^2(\eta+t)^2}.
\end{multline*}
The previous estimates together prove
\begin{equation}\label{eqn:interm11}
\int_0^{t\wedge\tau}|s_u(z_{t-u})-m(z_{t-u})|\left|\partial_z f_u(z_{t-u})-\frac{\im f_0(z_t)}{\im s_0(z_t)}\partial_z s_u(z_{t-u})\right|\rd u\leq
\varphi^C\left(\frac{(Nt)^a}{N^2t\eta}+\frac{1}{Nt}\right).
\end{equation}
Similarly we obtain 
\begin{equation}\label{eqn:interm2}
\int_0^{t\wedge\tau}\frac{1}{N}\left|\partial_{zz} f_u(z_{t-u})-\frac{\im f_0(z_t)}{\im s_0(z_t)}\partial_{zz} s_u(z_{t-u})\right|\rd u
\leq
\varphi^C\left(\frac{(Nt)^a}{N^2t\eta}+\frac{1}{Nt}\right).
\end{equation}
We now bound the bracket of the stochastic integral in (\ref{eqn:heq}): 
$$
\Big\langle
\int_0^{\cdot}\frac{e^{-u/2}}{\sqrt{N}}\sum_{q=1}^N\frac{{\frak r}_q({u})}{(z_{t-u}-x_q(u))^2}\rd B_q(u)
\Big\rangle_{t\wedge\tau}=\int_0^{t\wedge\tau}\frac{1}{N}\sum_{\hat q\leq \alpha N}\frac{|{\frak r}_q(u)|^2}{|z_{t-u}-\gamma_q|^4}\rd u
+
\int_0^{t\wedge\tau}\frac{1}{N}\sum_{\hat q> \alpha N}\frac{|{\frak r}_q(u)|^2}{|z_{t-u}-\gamma_q|^4}\rd u.
$$
For the contribution of the edge indices, we have
$$
\int_0^{t\wedge\tau}\frac{1}{N}\sum_{\hat q\leq \alpha N}\frac{|{\frak r}_q(u)|^2}{|z_{t-u}-\gamma_q|^4}\rd u\leq
\varphi^C \int_0^{t\wedge\tau}\frac{1}{N}\sum_{q=1}^N(N^{-\frac{2}{3}}(\hat q)^{-\frac{1}{3}})^2\rd u\leq \varphi^C \frac{t}{N^2}.
$$
For the bulk indices, we use
(\ref{eqn:rkbound}) for small $u$ and both (\ref{eqn:rkbound}) and (\ref{eqn:rkbound2}) for large $u$: 
\begin{multline}
\int_0^{t\wedge\tau}\frac{1}{N}\sum_{\hat q> \alpha N}\frac{|{\frak r}_q(u)|^2}{|z_{t-u}-\gamma_q|^4}\rd u.\leq
\varphi^C\int_0^{t\wedge\eta}\frac{1}{N}\sum_q\frac{N^{-2}}{|z_{t-u}-\gamma_q|^4}\rd u\\
+
\varphi^C\int_{t\wedge\eta}^{t\wedge\tau}\frac{1}{N}\sum_q\frac{1}{|z_{t-u}-\gamma_q|^4}
\min\left(\frac{|\gamma_q-E|^2}{N^2u^2}+g^2,\frac{1}{N^2}\right)
\rd u.\label{eqn:lililolo}
\end{multline}
The first integrand on the right-hand side above is $\OO(N^{-2}/(\eta+t-u)^3)$, so that the corresponding integral is $\OO(1/(Nt)^2)$.
For the second integral, we can assume $\eta<t$ and use $\min(a+b,c)\leq \min(a,c)+\min(b,c)$ for positive $a,b,c$, We first bound the contribution from $g$:
\begin{equation}\label{eqn:estg}
\int_{\eta}^{t}\frac{\rd u}{N}\left(\sum_q\frac{1}{|z_{t-u}-\gamma_q|^4}\right)\min\left(g^2,\frac{1}{N^2}\right)
\leq
\int_{\eta}^{t}\frac{\rd u}{(\eta+t-u)^3}\left(
\min\left(\frac{|u-t|^2}{N^2u^2},\frac{1}{N^2}\right)
+\frac{(Nu)^{2a}}{N^4u^2}
+\frac{\eta^2}{N^2t^2}
\right).
\end{equation}
We bound the term involving $\min\left(\frac{|u-t|^2}{N^2u^2},\frac{1}{N^2}\right)$ with
$$
\int_{\eta}^{t}\frac{\rd u}{(\eta+t-u)^3}\frac{|u-t|^2}{N^2u^2}\mathds{1}_{|u-t|<u}
\leq\frac{1}{N^2t^2}\int_\eta^t\frac{(u-t)^2}{(\eta+t-u)^3}\rd u\leq \frac{\log N}{N^2t^2},\ \  \ \ \  \ \ 
\int_{\eta}^{t}\frac{\rd u}{(\eta+t-u)^3}\frac{1}{N^2}\mathds{1}_{|u-t|>u}
\leq \frac{1}{N^2t^2}.$$
For the remaining terms from (\ref{eqn:estg}), we calculate
$$
\int_{\eta}^{t}\frac{\rd u}{(\eta+t-u)^3}\frac{(Nu)^{2a}}{N^4u^2}\leq\frac{(Nt)^{2a}}{N^4t^2\eta^2},\  \ \ \ \ \ \ \ 
\int_{\eta}^{t}\frac{\rd u}{(\eta+t-u)^3}
\frac{\eta^2}{N^2t^2}\leq\frac{1}{N^2t^2}.
$$
Finally, the contribution from $\min\left(\frac{|\gamma_q-E|^2}{N^2u^2},\frac{1}{N^2}\right)$ in (\ref{eqn:lililolo}) is bounded by
\begin{align*}
\int_{\eta}^{t}\frac{\rd u}{N}\sum_{q:|\gamma_q-E|<u}\frac{1}{|z_{t-u}-\gamma_q|^4}
\frac{|\gamma_q-E|^2}{N^2u^2}
&\leq \int_\eta^t\frac{\rd u}{N^2u^2}\int_{|x|<u}\frac{x^2}{x^4+(\eta+t-u)^4}\\
&\leq
\int_\eta^t\frac{\rd u}{N^2u^2}\left(
\frac{u^3}{t^4}\mathds{1}_{u<t/10}+
\frac{1}{\eta+t-u}\mathds{1}_{u>t/10}
\right)\leq\frac{1}{N^2t^2},\\
&\leq \int_{\eta}^{t}\frac{\rd u}{N^3}\left(
\frac{N}{t^3}\mathds{1}_{u<t/10}+
\frac{N}{u^3}\mathds{1}_{u>t/10}
\right)
\leq \frac{1}{N^2t^2}.
\end{align*}
The above estimates together prove
$$
\Big\langle
\int_0^{\cdot}\frac{e^{-u/2}}{\sqrt{N}}\sum_{q=1}^N\frac{{\frak r}_q({u})}{(z_{t-u}-x_q(u))^2}\rd B_q(u)
\Big\rangle_{t\wedge\tau}
\leq 
\varphi^C\left(\frac{(Nt)^{2a}}{N^4t^2\eta^2}
+\frac{1}{N^2t^2}
\right)
$$
for some $C$ independent of our choice of $\ell,m,p$.
By (\ref{eqn:bracket}) and a union bound we conclude that for any $D>0$ there exists $C>0$ such that 
$$
\mathbb{P}
\left(
\sup_{\ell,m,p,0\leq s\leq t\wedge\tau,|E_m|<2-\kappa}\left|
\int_0^{s}\frac{e^{-u/2}}{\sqrt{N}}\sum_{q=1}^N\frac{{\frak r}_q({u})}{(z_{t-u}-x_q(u))^2}\rd B_q(u)\right|
\leq 
\varphi^C\left(\frac{(Nt)^{a}}{N^2t\eta}
+\frac{1}{Nt}
\right)\right)
\geq 1-N^{-D}.
$$
Together with (\ref{eqn:interm11}) and  (\ref{eqn:interm2}), this concludes the proof that
$$
\mathbb{P}
\left(
\sup_{\ell,m,p,0\leq s\leq t\wedge\tau,|E_m|<2-\kappa}\left|
h_{s}(z)
\right|
\leq 
\varphi^C\left(\frac{(Nt)^{a}}{N^2t\eta}
+\frac{1}{Nt}
\right)\right)
\geq 1-N^{-D}.
$$
Remember that $\mathbb{P}(\min(\tau_0,\tau_1,\tau_2)=1)>1-N^{-D}$ by Lemma \ref{rig}, (\ref{eqn:ub}) and assumption (${\rm P}_a$). Together with the above equation, this implies (\ref{eqn:inter44}) and concludes the proof.
\end{proof}

We now can complete the proof of Lemma \ref{lem:averageMod}. 
As previously,  we fix $\alpha,D>0$ and find $N_0$ such that the conclusion of lemmas \ref{lem:shortRange},  \ref{lem:short} and \ref{lem:improved} hold with probability at least $1-N^{-D}$ for $N>N_0$, together with the rigidity estimate from Lemma \ref{rig}. We work on this good event, i.e. we assume that we are in $\mathscr{A}$ from (\ref{eqn:A}), and that (\ref{eqn:shortRange}), (\ref{eqn:short}) and (\ref{eqn:lemres}) hold.
We prove Lemma \ref{lem:averageMod} for $j,k\in\llbracket 2\alpha N,(1-2\alpha)N\rrbracket$, without loss of generality up to changing our initial choice of $\alpha$ into $\alpha/2$.

We rewrite the left-hand side of (\ref{eqn:averageMod}) as (i)+(ii)+(iii) and bound independently these terms defined as
\begin{align*}
(\rm i)&=\frac{1}{N}\im\sum_{|i-j|\leq \ell}\frac{({\rm U}_{\mathscr{S}}(u,s){\rm Av}{\frak u}(u)-{\rm Av}{\rm U}_{\mathscr{S}}(u,s){\frak u}(u))(i)}{x_i-z},\\
(\rm ii)&=\frac{1}{N}\im\sum_{|i-j|\leq \ell}\frac{({\rm Av}{\rm U}_{\mathscr{S}}(u,s){\frak u}(u)-{\rm Av}{\rm U}_{\mathscr{B}}(u,s){\frak u}(u))(i)}{x_i-z},\\
(\rm iii)&=\frac{1}{N}\im\sum_{|i-j|\leq \ell}\frac{({\rm Av}{\rm U}_{\mathscr{B}}(u,s){\frak u}(u))(i)-\bar{\frak u}_k(s)}{x_i-z}.
\end{align*}

We first estimate the numerator in (i),
$$
({\rm U}_{\mathscr{S}}(u,s){\rm Av}{\frak u}(u)-{\rm Av}{\rm U}_{\mathscr{S}}(u,s){\frak u}(u))(i)
=
\frac{1}{|\llbracket Nr,2Nr\rrbracket |}\sum_{Nr\leq b\leq 2Nr}\left(
{\rm U}_{\mathscr{S}}(u,s){\rm Flat}_b{\frak u}(u)
-
{\rm Flat}_b{\rm U}_{\mathscr{S}}(u,s){\frak u}(u)
\right)(i).
$$
If $|i-k|<b-\varphi\ell$, then
$
({\rm Flat}_b{\rm U}_{\mathscr{S}}(u,s){\frak u}(u))(i)=({\rm U}_{\mathscr{S}}(u,s){\frak u}(u))(i)
$
and by (\ref{eqn:short}) we have $({\rm U}_{\mathscr{S}}(u,s){\rm Flat}_b{\frak u}(u))(i)=({\rm U}_{\mathscr{S}}(u,s){\frak u}(u))(i)+\OO(N^{-D})$, so that in this case
\begin{equation}\label{eqn:neg}
\left(
{\rm U}_{\mathscr{S}}(u,s){\rm Flat}_b{\frak u}(u)
-
{\rm Flat}_b{\rm U}_{\mathscr{S}}(u,s){\frak u}(u)
\right)(i)
=
\OO(N^{-D}).
\end{equation}

If $|i-k|>b+\varphi\ell$, then
$
({\rm Flat}_b{\rm U}_{\mathscr{S}}(u,s){\frak u}(u))(i)=\bar{\frak u}_k(t)
$
and, again with (\ref{eqn:short}),  $({\rm U}_{\mathscr{S}}(u,s){\rm Flat}_b{\frak u}(u))(i)=\bar{\frak u}_k(t)+\OO(N^{-D})$, so that (\ref{eqn:neg}) also holds in this case.

Assume now $|i-k|\in[b-\varphi\ell,b+\varphi\ell]$. By using (\ref{eqn:short}) first and then (${\rm P}_a$) and Lemma \ref{lem:init}, we have
\begin{multline*}
|({\rm U}_{\mathscr{S}}(u,s){\rm Flat}_b{\frak u}(u)-{\rm Flat}_b{\rm U}_{\mathscr{S}}(u,s){\frak u}(u))(i)|\leq \max_{m:||m-k|-b|\leq 2\varphi \ell}|{\frak u}_m(s)-\bar{\frak u}_k(t)|+\OO(N^{-D})\\
\leq 
\max_{||m-k|-b|\leq 2\varphi \ell}|{\frak u}_m(s)-\bar{\frak u}_m(s)|+
\max_{||m-k|-b|\leq 2\varphi \ell}|\bar{\frak u}_m(s)-\bar{\frak u}_k(t)|+
\OO(N^{-D})
\leq \varphi^C\left(\frac{(Nt)^a}{N^2t}+\frac{r+|t-u|}{Nt}\right).
\end{multline*}
We conclude that
\begin{equation}\label{eqn:i}
({\rm i})=\OO(\varphi^C)\frac{\ell}{Nr}\left(\frac{(Nt)^a}{N^2t}+\frac{r+|t-u|}{Nt}\right).
\end{equation}

We now estimate  (ii). As $|i-j|\leq \ell$, we have $i\in\llbracket \alpha N,(1-\alpha)N\rrbracket$ and  (\ref{eqn:shortRange}) applies: we obtain
\begin{multline*}
|({\rm Av}{\rm U}_{\mathscr{S}}(u,s){\frak u}(u)-{\rm Av}{\rm U}_{\mathscr{B}}(u,s){\frak u}(u))(i)|\leq 
|({\rm U}_{\mathscr{S}}(u,s){\frak u}(u)-{\rm U}_{\mathscr{B}}(u,s){\frak u}(u))(i)|\leq
\varphi^C|u-t|
\left(
\frac{N}{\ell}\frac{(Nt)^a}{N^2t}
+
\frac{1}{Nt}
\right),
\end{multline*}
where the first inequality follow from (\ref{eqn:coefprime}).
The same bound for an average over $i$ gives
\begin{equation}\label{eqn:ii}
({\rm ii})=\OO(\varphi^C)|u-t|
\left(
\frac{N}{\ell}\frac{(Nt)^a}{N^2t}
+
\frac{1}{Nt}
\right).
\end{equation}
Finally, to estimate (iii), we  use (\ref{eqn:coefprime}) to first decompose
\begin{multline}\label{eqn:iii}
({\rm iii})=
a_j\sum_{i=1}^N\frac{1}{N}\im\frac{{\frak u}_i(s)-\bar{\frak u}_k(s)}{x_i-z}
-
\frac{a_j}{N}\im\sum_{|i-j|> \ell}\frac{{\frak u}_i(s)-\bar{\frak u}_k(s)}{x_i-z}
+
\frac{1}{N}\im\sum_{|i-j|\leq \ell}\frac{(a_i-a_j)({\frak u}_i(s)-\bar{\frak u}_k(s))}{x_i-z}\\
+
\frac{1}{N}\im\sum_{|i-j|\leq \ell}\frac{(1-a_i)(\bar{\frak u}_k(t)-\bar{\frak u}_k(s))}{x_i-z}.
\end{multline}
The first sum is also (we use (\ref{elem2}) for the first equality and the main estimate (\ref{eqn:lemres}) for the second equality below)
\begin{equation}\label{eqn:iii1}
\frac{e^{s/2}}{N}\im f_s(z)-\im s_s(z)\bar {\frak u}_k(s)
=\frac{e^{s/2}}{N}\im f_s(z)-\im s_s(z)\frac{\im f_0(z_s)}{N\im s_0(z_s)}+
\OO(\varphi)\frac{\eta+r}{Nt}
=
\OO(\varphi^C)\left(\frac{(Nt)^a}{N^3t\eta}
+\frac{\eta+r}{Nt}
\right).
\end{equation}
To bound the second  sum in (\ref{eqn:iii}), for $i$ in the bulk we write
\begin{equation}\label{eqn:estu}
|{\frak u}_i(s)-\bar{\frak u}_k(s)|
\leq
|{\frak u}_i(s)-\bar{\frak u}_i(s)|
+
|\bar{\frak u}_i(s)-\bar{\frak u}_k(s)|
\leq
\frac{(Nt)^a}{N^2t}+\frac{|i-k|}{N^2t}
\end{equation} (for $i$ at the edge we can use (\ref{eqn:ub}) which gives a negligible contribution), and obtain the estimate
\begin{multline}
\frac{1}{N}\sum_{|i-j|>\ell}\frac{\eta}{\eta^2+(\gamma_i-\gamma_j)^2}\frac{(Nt)^a}{N^2t}\label{eqn:iii2}
+
\frac{1}{N}\sum_{|i-j|>\ell}\frac{\eta}{\eta^2+(\gamma_i-\gamma_j)^2}\frac{|i-j|}{N^2t}
+
\frac{1}{N}\sum_{|i-j|>\ell}\frac{\eta}{\eta^2+(\gamma_i-\gamma_j)^2}\frac{Nr}{N^2t}\\
\leq
\frac{(Nt)^a}{N^2t}\frac{N\eta}{\ell}+\frac{r}{Nt}+\frac{\eta}{Nt}. 
\end{multline}
The third sum in (\ref{eqn:iii}), we use (\ref{eqn:coefprime}) and (\ref{eqn:estu}) to obtain
\begin{equation}\label{eqn:iii3}
\frac{1}{N}\im\sum_{|i-j|\leq \ell}\frac{(a_i-a_j)({\frak u}_i(s)-\bar{\frak u}_k(s))}{x_i-z}=\OO\left(\frac{\ell}{Nr}\right)\left(\frac{(Nt)^a}{N^2t}+\frac{r}{Nt}\right).
\end{equation}
The fourth sum in  (\ref{eqn:iii}) is bounded by (\ref{elem1}), which added to the error estimates (\ref{eqn:i}), (\ref{eqn:ii}), (\ref{eqn:iii1}), (\ref{eqn:iii2}), (\ref{eqn:iii3}) concludes the proof.

\section{Extreme gaps}

\subsection{Reverse heat flow.}\ 
We first state a quantitative analogue of \cite[Proposition 4.1]{ErdSchYau2011}.
This reverse heat flow argument first appeared in \cite{ErdPecRamSchYau2010}.
Its proof is essentially the same as in \cite{ErdSchYau2011}. In the following $\rd \gamma$ denotes the standard Gaussian measure which is reversible for the Ornstein-Uhlenbeck dynamics with generator $A=\frac{1}{2}\partial_{xx}-\frac{x}{2}\partial_x$.

\begin{lemma}\label{lem:RevHeat}
Let $0<2a<b<1$. Assume $e^{-V}$ is a centered probability density, with $V$ smooth on scale $\sigma=N^{-a}$ in the sense of (\ref{eqn:smooth}) and $\int_{[-x,x]^c} e^{-V(y)}\rd y\leq {\theta^{-1}} e^{-x^\theta}$ for some $\theta>0$. 
Denote $u=\rd e^{-V}/\rd\gamma$.

Let $t=N^{-b}$. Then for any $D>0$ there exists  $C>0$ and a probability density $g_t$ w.r.t. $\gamma$ such that

\begin{enumerate}
\item $
\int|e^{tA}g_t-u|\rd \gamma\leq C N^{-D},
$
\item $g_t\rd \gamma$ is centered, has same variance as $u\rd\gamma$, and satisfies  $\int_{[-x,x]^c} g_t\rd\gamma\leq \theta^{-1} e^{-x^\theta}$ for some $\theta>0$.
\end{enumerate}
\end{lemma}

\begin{proof} Let $\alpha=\alpha(N)>0$ to be chosen, $\theta_0$ is a smooth cutoff function equal to $1$ on $[-1,1]$ and 
$0$ on $[-2,2]^c$, and
$\theta(x)=\theta_0(\alpha x)$. We define 
$$
h_t=u+\theta \xi_t,\ {\rm with}\ \xi_t=\left(
-t A+\frac{1}{2}t^2A^2+\dots+(-1)^{k-1}\frac{t^{k-1}}{(k-1)!}A^{k-1}
\right)u.
$$
Using (\ref{eqn:smooth}), for any $k>0$ there exists $C>0$ such that
\begin{equation}\label{eqn:thetaxi}
|\theta\xi_t|\leq 
C_k\sum_{\ell=1}^{k-1}t^\ell\sigma^{-2\ell}\alpha^{-C_k}u.
\end{equation}
The function $h_t$ is therefore positive if $\alpha=N^{-\e}$ with $0<\e<(b-2a)/C_k$.

Moreover, from \cite[Equation (4.4)]{ErdSchYau2011}, we have
$$
\int|e^{tA}h_t-u|\rd\gamma\leq C_k\int_0^t\left(t^k\int|A^k u|\rd\gamma+
\int|A(\theta-1)\xi_s|\rd\gamma+
\int|(\theta-1)\partial_s\xi_s|\rd\gamma
\right)\rd s.
$$
Still using (\ref{eqn:smooth}), we easily have 
$
t^k\int|A^k u|\rd\gamma\leq C_k t^k\sigma^{-2k}
$
and
$$
\int_0^t|A(\theta-1)\xi_s|\rd\gamma\rd s
+
\int_0^t|(\theta-1)\partial_s\xi_s|\rd\gamma\rd s
\leq
C t \sigma^{-2k} \int_{[-\alpha^{-1},\alpha^{-1}]^c}(1+|x|)^{C_k}u\rd\gamma\leq 
C t \sigma^{-2k}e^{-\alpha^{-\wt c}}
$$
for some $\wt c>0$, where we used the tail assumption $\int_{[-x,x]^c} e^{-V}\leq c e^{-x^c}$.

All together, for $k$ large enough (depending on $D$) and $0<\e<(b-2a)/C_k$, we obtain 
$$
\int|e^{tA}h_t-u|\rd \gamma\leq C N^{-D}.
$$
Moreover, from (\ref{eqn:thetaxi}) and our choice of parameters we have $c_t:=\int h_t\rd\gamma=1+\OO(N^{-D})$, so that $g_t:=h_t/c_t$ (now a probability density) also satisfies 
$\int|e^{tA}g_t-u|\rd \gamma\leq C N^{-D}$. Similarly, by a dilation with factor $1+\OO(N^{-D})$, $g_t$ can be dilated into a probability with variance 1. Finally, $(ii)$ easily follows from $(\ref{eqn:thetaxi})$ and the hypothesis $\int_{[-x,x]^c} e^{-V}\leq {\theta^{-1}} e^{-x^\theta}$.
\end{proof}

\subsection{Proof of theorems \ref{thm:smallprocess} and \ref{thm:largest}.}\ 
We illustrate this classical reasoning with  Theorem \ref{thm:largest}, Theorem \ref{thm:smallprocess} being proved similarly based on Corollary \ref{cor:gaps} and Lemma \ref{lem:RevHeat}.

We assume $H$ is smooth on scale $\sigma$. From Lemma \ref{lem:RevHeat}, there exists a generalized Wigner matrix $\wt H$ such that if $\wt H_t$ denotes  its evolution under the Dyson Brownian Motion dynamics with initial condition $\wt H$, the total variation distance between $\wt H_t$ and
$H$ is of order $N^{-D}$ for any $D$, provided $t\leq N^{-\e}\sigma^{2}$.
In particular,  the total variation distance between their spectra is also at most $N^{-D}$, and
$
{\rm d}_{\rm TV}\left(\tau_k^*(\wt H_t),\tau_k^*(H)\right)\leq N^{-D},
$
so that for large enough $N$ we have
\begin{multline*}
{\rm d}_{\rm W}\left(\tau_k^*(\wt H_t),\tau_k^*(H)\right)\leq\int_{-N^5}^{N^5}\rd x|\mathbb{P}(\tau_k^*(\wt H_t)\leq x)-\mathbb{P}(\tau_k^*(H)\leq x)|\\
+\mathbb{E}(\tau_k^*(\wt H_t)\mathds{1}_{|\tau_k^*(\wt H_t)|>N^5})
+\mathbb{E}(\tau_k^*(H)\mathds{1}_{|\tau_k^*(H)|>N^5})\leq C N^{-D+5}.
\end{multline*}

On the other hand,  for such $t$, from Corollary \ref{cor:gaps}  the gaps between bulk eigenvalues
of $\wt H_t$ can all be coupled with some GUE gaps with some error $N^\e/(N^2t)$. With the third characterization of the Wasserstein distance in (\ref{eqn:Wass}), we obtain
$${\rm d}_{\rm W}(\tau_k^*(\wt H_t),\tau_k^*({\rm GUE}))\leq \frac{N^{\varepsilon}}{N t}.$$
The two equations above conclude the proof.

\begin{remark}
From the above proof, it is clear that if uniform (in $N$) boundedness of the density of $\tau_k(GOE)$ or $\tau_k^*(GOE)$ was known, then the rates of convergence in Corollary \ref{cor:smallest} and Theorem \ref{thm:largest} would also hold for the Kolmogorov-Smirnov distance. It is not obvious that the methods in \cite{BenBou2013,FengWei2018I,FengWei2018II} give this boundedness, as they rely on moments calculations. 
\end{remark}

\section{Rate of convergence to Tracy-Widom}

\subsection{Proof of Theorem \ref{thm:rate}.}\ 
This rate of convergence relies on a main result of this paper, Theorem \ref{thm:edgerelaxation}, and the following Proposition \ref{t3}, 
a quantitative version of the Green's function comparison theorem from \cite{ErdYauYin2012Univ}. It is proved exactly in the same way, after carefully keeping track of all error terms.
For completeness, we give the proof in the next subsection.

For the statement, we consider a scale $\rho=\rho(N)\in[N^{-1},N^{-\frac{2}{3}}]$, and  a function $f=f(N):\mathbb{R}\to\mathbb{R}$ satisfying
$$
\|f^{(k)}\|_\infty\leq C_k \rho^{-k}, \ 0\leq k\leq 2.
$$
Assume also that $f$ is non-decreasing, $f(x)\equiv 0$ for $x<E$, $f(x)\equiv 1$ for $x>E+\rho$, with $|E-2|<\varphi N^{-2/3}$.
Moreover, let $F$ be a fixed smooth non-increasing function  such that $F(x)=1$ for $x\leq 0$, $F(x)=0$ for $x\geq 1$.

\begin{proposition} \label{t3}
There exists $C>0$ such  that the following holds. Let $H^{\f v}$ and $H^{\f w}$ be generalized Wigner ensembles satisfying (\ref{eqn:tail}). 
Assume that the first three moments of the entries ($h_{ij}=\sqrt{N}H_{ij}$) are the same, i.e.
$
\E^{\f v}(h_{ij}^k)=  \E^{\f w}(h_{ij}^k)
$
for all $1\leq i\leq j\leq N$ and $1\leq k\leq 3$.
Assume also that for some parameter $t=t(N)$ we have
$$
	\Big |  \E^{\f v}(h_{ij}^4)- \E^{\f w}(h_{ij}^4)\Big | \le t
\for i \leq j.
$$
With the above notations for  the test functions $f,F$, we have
$$
 \left|\left(\E^{\f v} - \E^{{\f w}}\right) F\left(
 {\rm Tr}f(H)
 \right)\right| \leq \varphi^C\left(\frac{t}{N\rho}+\frac{1}{(N\rho)^2}+\frac{1}{N}\right).
$$
\end{proposition}

We now can complete the proof of Theorem \ref{thm:rate}.
Let $x\in\mathbb{R}$. If $|x|>\varphi$, then for any $D>0$ we have $\mathbb{P}_H(N^{2/3}(\lambda_N-2)\leq x)=\mathbb{P}({\rm TW}_1\leq x)+\OO(N^{-D})$ for large enough $N$. So we now assume 
$|x|<\varphi$.

Define a non-decreasing $f_1$ such that
$f_1(x)=1$ for $x>2+xN^{-2/3}$, $f_1(x)=0$  for $x<2+xN^{-2/3}-\rho$. We also denote $f_2(x)=f_1(x-\rho)$.
We then  have
\begin{equation}\label{eqn:bound1}
\E_HF\left(\sum_{i=1}^N f_1(\lambda_i)\right)\leq\mathbb{P}_{H}\left(\lambda_N<2+xN^{-2/3}\right)\leq \E_H F\left( \sum_{i=1}^N f_2(\lambda_i)\right).
\end{equation}
To understand the above right-hand side, if $\lambda_N<2+xN^{-2/3}$ then
$\sum_{i=1}^N f_2(\lambda_i)=0$ so that $F\left( \sum_{i=1}^N f_2(\lambda_i)\right)=1$; the inequality on the left follows by a similar argument.

Moreover, as is classical and mentioned in the proof of Theorem \ref{Gaussian}, we can find
 a generalized  Wigner matrix  $\wt H_0$ 
such that
the
Gaussian divisible ensemble $\wt H_t:=e^{-t/2}\wt H_0+(1-e^{-t})^{1/2}\ U$, 
($U$ is an independent standard  GOE matrix)  has its first
three moments which
match exactly those of the matrix $H$ and the differences between 
the fourth moments of the two ensembles is
$\OO(t)$ (see  for example by \cite[Lemma 3.4]{EYYBernoulli}). By applying Proposition  \ref{t3}, the bound (\ref{eqn:bound1}) becomes
\begin{multline*}
\E_{\wt H_t}F\left(\sum_{i=1}^N f_1(\lambda_i)\right)-\varphi^C\left(\frac{t}{N\rho}+\frac{1}{(N\rho)^2}+\frac{1}{N}\right)\leq\mathbb{P}_{H}\left(\lambda_N<2+xN^{-2/3}\right)\\
\leq \E_{\wt H_t} F\left(\sum_{i=1}^N f_2(\lambda_i)\right)+\varphi^C\left(\frac{t}{N\rho}+\frac{1}{(N\rho)^2}+\frac{1}{N}\right).
\end{multline*}
Using again  (\ref{eqn:bound1}) but now for the ensemble $\wt H_t$ and  for $f_1,f_2$ shifted by $\pm\rho$, the previous equation gives
\begin{multline*}
\mathbb{P}_{\wt H_t}\left(\lambda_N<2+xN^{-2/3}-\rho\right)-\varphi^C\left(\frac{t}{N\rho}+\frac{1}{(N\rho)^2}+\frac{1}{N}\right)\leq\mathbb{P}_{H}\left(\lambda_N<2+xN^{-2/3}\right)\\
\leq \mathbb{P}_{\wt H_t}\left(\lambda_N<2+xN^{-2/3}+\rho\right)+\varphi^C\left(\frac{t}{N\rho}+\frac{1}{(N\rho)^2}+\frac{1}{N}\right).
\end{multline*}
When combined with the edge relaxation  Theorem \ref{thm:edgerelaxation}, this estimate gives
\begin{multline}\label{eqn:bound2}
\mathbb{P}_{\rm GOE}\left(N^{\frac{2}{3}}(\lambda_N-2)<x-N^{\frac{2}{3}} \rho-\frac{N^\e}{N^{\frac{1}{3}}t}\right)-\varphi^C\left(\frac{t}{N\rho}+\frac{1}{(N\rho)^2}+\frac{1}{N}\right)\leq\mathbb{P}_{H}\left(N^{\frac{2}{3}}(\lambda_N-2)<x\right)\\
\leq \mathbb{P}_{\rm GOE}\left(N^{\frac{2}{3}}(\lambda_N-2)<x+N^{\frac{2}{3}}\rho+\frac{N^\e}{N^{\frac{1}{3}}t}\right)+\varphi^C\left(\frac{t}{N\rho}+\frac{1}{(N\rho)^2}+\frac{1}{N}\right).
\end{multline}
Moreover, from \cite{JohMa2012} uniformly in $|x|<\varphi$ we have 
$$
\mathbb{P}_{\rm GOE}\left(N^{\frac{2}{3}}(\lambda_N-2)<x\right)=\mathbb{P}\left({\rm TW}_1<x\right)+\OO(N^{-1/2})
$$
(more precisely the main result of  \cite{JohMa2012}  gives the better error of order $N^{-2/3}$, but only for $x>-C$, and a straightforward adaptation of the proof shows the above bound).
By using this GOE result and boundedness of the density of ${\rm TW}_1$ in (\ref{eqn:bound2}),
we obtain
$$
\mathbb{P}_{H}\left(N^{\frac{2}{3}}(\lambda_N-2)<x\right)-\mathbb{P}\left({\rm TW}_1<x\right)=\OO(N^\e)
\left(N^{\frac{2}{3}}\rho+\frac{1}{N^{\frac{1}{3}}t}+\frac{t}{N\rho}+\frac{1}{(N\rho)^2}+\frac{1}{\sqrt{N}}\right).
$$
The optimal bound $N^{-2/9+\epsilon}$ is obtained for $t=N^{-1/9}$ and $\rho=N^{-8/9}$. This concludes the proof.

\subsection{Proof of Proposition \ref{t3}.}\ We closely follow the notations and reasoning from \cite[Theorem 17.4]{ErdYau2017}. 
We first fix a bijective ordering map of the index set of the independent matrix entries, $\phi:\{(i,j):1\leq i\leq j\leq N\}\to\llbracket1,\gamma(N)\rrbracket$, with $\gamma(N)=N(N+1)/2$.
Then let $H_\gamma$ be the generalized Wigner matrix whose matrix elements $h_{ij}$ follow the ${\f v}$-distribution for $\phi(i,j)\leq \gamma$, and the ${\f w}$-distribution otherwise, so that
$H^{\f v}=H_0$ and $H^{\f w}=H_{\gamma(N)}$.
By summation, it is sufficient to prove that uniformly in $\gamma\in\llbracket 1,\gamma(N)\rrbracket$ we have
\begin{equation}\label{eqn:sufficient1}
 \left|\E F\left(
 {\rm Tr}f(H_\gamma)
 \right) - \E F\left(
 {\rm Tr}f(H_{\gamma-1})\right) \right| \leq \varphi^CN^{-2}\left(\frac{t}{N\rho}+\frac{1}{(N\rho)^2}+\frac{1}{N}\right).
\end{equation}
 
Let $\chi$ be a smooth, symmetric function such that $\chi(y)=1$ if $|y|<N^{-2/3}$, $\chi(y)=0$ if $|y|>2N^{-2/3}$, $\|\chi'\|_\infty<N^{2/3}$. 
With the Helffer-Sj{\H o}strand formula, if the $\lambda_i$'s are the eigenvalues of a  matrix $H$, we have
$$
\sum f(\lambda_i)=\int_{\mathbb{C}}g(z){\rm Tr}\frac{1}{H-z}\rd^2 z,\ g(z):=\frac{1}{\pi}(\ii y f''(x)\chi(y)+\ii(f(x)+\ii y f'(x))\chi'(y)),\ z=x+\ii y,
$$
where $\rd^2 z$ is the Lebesgue measure on $\mathbb{C}$. We define
$
\Xi^H=\int_{|y|> N^{-1}}g(z){\rm Tr}(H-z)^{-1}\rd^2 z,
$
and first bound 
$$
\left|\sum f(\lambda_i)-\Xi^H\right|
\leq
\iint_{|y|<\frac{1}{N},E<x<E+\rho}|f''(x)|\sum_{i}\frac{y^2}{|\lambda_i-(x+\ii y)|^2}\rd x\rd y
\leq \int_{E<x<E+\rho}\frac{1}{\rho^2N^3}\sum_{i}\frac{\rd x}{|\lambda_i-(x+\frac{\ii}{N})|^2},
$$
where for the last inequality we used $y^2|\lambda-(x+\ii y)|^{-2}\leq N^{-2}|\lambda-(x+\ii/N)|^{-2}$.

If $i\geq N-\varphi^C$, we simply bound $\int_{\mathbb{R}} |\lambda_i-(x+\ii/N)|^{-2}\rd x\leq CN$.
If $i\leq N-\varphi^C$, with overwhelming probability we have 
$\int_{E<x<E+\rho} |\lambda_i-(x+\ii/N)|^{-2}\rd x\leq \rho|E-\gamma_i|^{-2}$. We therefore have
\begin{equation}\label{eqn:app}
\left|\sum f(\lambda_i)-\Xi^H\right|\leq \frac{\varphi^C}{(N\rho)^2} +\frac{\varphi^C}{\rho^2N^3}\sum_{i\leq N-\varphi^C} \frac{\varphi^C\rho}{|E-\gamma_i|^2}\leq \frac{\varphi^C}{(N\rho)^2}\left(1+\frac{\rho}{N}\sum_{k\geq 1}\frac{1}{(k/N)^{4/3}}\right)
=\OO\left(\frac{\varphi^C}{(N\rho)^2}\right)
\end{equation}
with overwhelming probability. 

As (\ref{eqn:app}) holds,  (\ref{eqn:sufficient1}) will be true provided that  uniformly in  $\gamma\in\llbracket 1,\gamma(N)\rrbracket$, we have
\begin{equation}\label{eqn:sufficient2}
 \left|\E F\left(
\Xi^{H_\gamma}
 \right) - \E F\left(
 \Xi^{H_{\gamma-1}}\right) \right| \leq \varphi^CN^{-2}\left(\frac{t}{N\rho}+\frac{1}{(N\rho)^2}+\frac{1}{N}\right).
\end{equation}
For this fixed $\gamma$ corresponding to $(i,j)$ ($\phi(i,j)=\gamma$), we can write
$$
H_{\gamma-1}=Q+\frac{1}{\sqrt{N}}V,\ H_{\gamma}=Q+\frac{1}{\sqrt{N}}W,\ 
$$
where $Q$ coincides with $H_{\gamma-1}$ and $H_{\gamma}$  except on the entries $(i,j)$ and $(j,i)$, where it is $0$.
We abbreviate
$$
R=\frac{1}{Q-z},\ S=\frac{1}{H_{\gamma}-z},\ \hat R=\frac{1}{N}{\rm Tr}R,\ \hat R^{(m)}_{\bv}=\frac{(-1)^m}{N}{\rm Tr}(RV)^mR,\ \Omega_{\bv}=-\frac{1}{N}{\rm Tr}(RV)^5S.
$$
Then the resolvent expansion at fifth order gives
$$
\frac{1}{N}{\rm Tr}S=\hat R+\xi_\bv,\ {\rm with}\ \xi_\bv=\sum_{m=1}^4 N^{-\frac{m}{2}}\hat R^{(m)}_\bv+N^{-\frac{5}{2}}\Omega_\bv.
$$
By Taylor expansion, we have
\begin{multline}
\E F\left(
\Xi^{H_\gamma}
 \right) - \E F\left(
 \Xi^{H_{\gamma-1}}\right)
=
\sum_{\ell=1}^3 \E \frac{F^{(\ell)}(\Xi^Q)}{\ell!}\left(( \Xi^{H_{\gamma}}-\Xi^Q)^\ell-( \Xi^{H_{\gamma-1}}-\Xi^Q)^\ell\right)\\
+
\OO(\|F^{(4)}\|_\infty)\left(
\E\left(( \Xi^{H_{\gamma}}-\Xi^Q)^4+( \Xi^{H_{\gamma-1}}-\Xi^Q)^4\right)
\right)\label{eqn:Taylor}.
\end{multline}
We first bound the above fourth order error term. For a matrix $M$ we denote
$\|M\|_\infty=\max_{i,j}|M_{ij}|$,
$\|M\|_\infty^{\rm off}=\max_{i\neq j}|M_{ij}|$ and
$\|M\|_\infty^{\rm diag}=\max_{i}|M_{ii}|$.
 
   By the first order resolvent expansion, we have (all integration domains are $|y|>N^{-1},|x|<3$, i.e. we omit the contribution from $x>3$, clearly negligible)
\begin{multline*}
|\Xi^{H_{\gamma}}-\Xi^Q|\leq N^{-1/2}\int|g(z)|\, \left|{\rm Tr}S(z)VR(z)\right|\rd^2 z\\
\leq \varphi^C N^{1/2}\int|g(z)|\,\|S(z)\|_{\infty}^{\rm off}\,\|R(z)\|_\infty\rd^2 z+
\varphi^C N^{-1/2}\int|g(z)|\,\|S(z)\|_{\infty}^{\rm diag}\,\|R(z)\|_\infty\rd^2 z,
\end{multline*}
with overwhelming probability, where we used the fact that $V$ has only two non-zero entries, of order 1.
The local law for Wigner matrices from \cite{ErdYauYin2012} states that
uniformly in any $z$ in a compact set, for any $D>0$,
\begin{equation}\label{eqn:local}
\mathbb{P}\left(\|S(z)-m(z){\rm Id}\|_\infty>\varphi^C\Psi(z)\right)\leq N^{-D},\ \ \Psi(z)=\frac{1}{Ny}+\sqrt{\frac{\im m(z)}{Ny}}.
\end{equation}
and the same estimate holds for $R(z)$.
From (\ref{eqn:imagM}) we note that $\Psi(y)<\varphi^C/(Ny)$ when $0<y<N^{-2/3}, |E-2|<\varphi^C N^{-2/3}$.
We conclude that
for any $D>0$  we have (the contribution of diagonal resolvent entries is negligible, and we can also a priori omit the domain $|x-2|>\varphi^{C}N^{-2/3}$ by rigidity)
$$
\Xi^{H_{\gamma}}-\Xi^Q=\OO(\varphi^C)N^{1/2}\int_{N^{-1}<|y|<N^{-2/3},|x-2|<\varphi^CN^{-\frac{2}{3}}}\frac{|g(z)|}{(Ny)^2}\rd^2 z=\OO\left(\frac{\varphi^C}{N^{3/2}\rho}\right)
$$
with probability at least $1-N^{-D}$ for $N>N_0(D)$. As a consequence, 
$\E\left(( \Xi^{H_{\gamma}}-\Xi^Q)^4\right)=\OO(\varphi^C/(N^6\rho^4))$ and the same result holds for $\E\left(( \Xi^{H_{\gamma-1}}-\Xi^Q)^4\right)$. The fourth order term in 
(\ref{eqn:Taylor}) can therefore be bounded with $\varphi^C N^{-2}(N\rho)^{-4}\leq \varphi^C N^{-2}(N\rho)^{-2}$ for $\rho\geq N^{-1}$.

Consider now the linear $\ell=1$ term in  (\ref{eqn:Taylor}).
We have
\begin{equation}\label{eqn:error1}
\E F'(\Xi^Q)\left(\Xi^{H_{\gamma}}- \Xi^{H_{\gamma-1}}\right)
=
\E F'(\Xi^Q)\int g(z)
\left(
\sum_{m=1}^4 N^{-\frac{m}{2}+1}(\hat R^{(m)}_\bv-\hat R^{(m)}_\bw)
+N^{-\frac{3}{2}}(\Omega_\bv-\Omega_\bw)\right)\rd^2 z.
\end{equation}
The first three moments associated to $\bv$ and $\bw$ match, so the cases $m=1,2,3$ in the above formula gives null contribution.

For $m=4$, as the fourth moments for $\bv$ and $\bw$ differ by $t$, we have ($\E$ below just refers to the expectation on $V$, $W$)
\begin{equation}\label{eqn:interm1}
\E N\left(\hat R^{(4)}_\bv-\hat R^{(4)}_\bw\right)
=
\E {\rm Tr}\left(
(RV)^4R-(RW)^4R
\right)
=\OO(Nt)(\max_{i\neq j}|R_{ij}|)^2(\max_i|R_{ii}|)^3,
\end{equation}
where we have used that in the expansion $${\rm Tr}
(RV)^4R=\sum_k\sum_{\{a_p,b_p\}=\{i,j\}}R_{ka_1}v_{a_1b_1}R_{b_1a_2}v_{a_2b_2}R_{b_2a_3}v_{a_3b_3}
R_{b_3a_4}v_{a_4b_4}R_{b_4k},
$$
typically we have $a_1\neq k$, $b_4\neq k$, but we may have  $b_1=a_2$, $b_2=a_3$, $b_3=a_4$. More precisely the contribution of $k$'s which are either $i$ or $j$ is  combinatorially negligible: we  omit this case here and in the following. 

As mentioned after (\ref{eqn:local}), in the domain $N^{-1}<y<2N^{-2/3},|x-2|<\varphi^CN^{-\frac{2}{3}}$ we have
$\max_{i\neq j}|R_{ij}|<\frac{\varphi^C}{Ny}$
and $\max_i|R_{ii}|<\varphi^C$, so that
$$
\E F'(\Xi^Q)\int g(z)
N^{-\frac{4}{2}+1}(\hat R^{(m)}_\bv-\hat R^{(m)}_\bw)
\rd^2 z
=
\OO\left(\varphi^C\frac{t}{N^2}\right)\int\frac{|g(z)|}{Ny^2}\rd^2 z=\OO\left(\frac{\varphi^C}{N^2}\frac{t}{N\rho}\right),$$
where all integration domains are $N^{-1}<|y|<N^{-2/3},|x-2|<\varphi^CN^{-\frac{2}{3}}$.

For the term $\Omega_{\bv}-\Omega_{\bw}$ in (\ref{eqn:error1}), we don't use any cancellation between $\bv$ and $\bw$.
As for (\ref{eqn:interm1}), an expansion and the local law give
$
\Omega_{\bv}=\OO(\varphi^C(Ny)^{-2})$,
so that
$$
\E F'(\Xi^Q)\int g
N^{-\frac{3}{2}}\Omega_\bv
\rd^2 z
=
\OO(\varphi^CN^{-\frac{3}{2}})\int\frac{|g(z)|}{(Ny)^2}\rd^2 z=
\OO\left(\frac{\varphi^C}{N^2}\frac{1}{N^{3/2}\rho}\right)
=
\OO\left(\frac{\varphi^C}{N^2}\right)\left(\frac{1}{(N\rho)^2}+\frac{1}{N}\right),
$$
where we integrate on $N^{-1}<|y|<N^{-2/3},|x-2|<\varphi^CN^{-\frac{2}{3}}$ again.

In sum, with the above two equations we proved that the $\ell=1$ term in (\ref{eqn:Taylor}) is bounded by the 
right-hand side of (\ref{eqn:sufficient2}). Similar perturbative expansions show that the $\ell=2,3$ contributions are of smaller order, similarly to the proof of \cite[Theorem 17.4]{ErdYau2017}. The detail are left to the reader. This concludes the proof of (\ref{eqn:sufficient2}) and of Proposition \ref{t3}.

\nc
\setcounter{equation}{0}
\setcounter{theorem}{0}
\renewcommand{\theequation}{A.\arabic{equation}}
\renewcommand{\thetheorem}{A.\arabic{theorem}}
\appendix
\setcounter{secnumdepth}{0}
\section[Appendix]{Appendix}

\begin{lemma}\label{lem:A1}
There is a universal constant $c$ such that for any $z=z_{t-s}$ as in (\ref{eqn:boundsup}), any $j$ and $k_1,k_2\in I_j$, we have
$$
c|z-{\gamma_{k_2}}|\leq |z-\gamma_{k_1}|\leq c^{-1}|z-\gamma_{k_2}|.
$$
\end{lemma}

\begin{proof}
For any $k\leq \frac{3}{4}N$ we have $\gamma_{k+\lfloor\varphi^2\rfloor}-\gamma_k\leq C \frac{\varphi^2}{N \kappa(E)^{1/2}}\leq C {\rm dist}(\gamma_k,\mathscr{S})$ (indeed $\gamma_{k+\lfloor \varphi^2\rfloor}-\gamma_k\leq C (N^{-1}\varphi^2)^{2/3}$ if $k\leq \varphi^2$
and $\gamma_{k+\lfloor \varphi^2\rfloor}-\gamma_k\leq C (N^{-2/3}\varphi^2)k^{-\frac{1}{3}}$ if $\varphi^2\leq k\leq \frac{3}{4}N$). 
This implies in particular that $|\gamma_{k_1}-\gamma_{k_2}|\leq C {\rm dist}(\gamma_{k_1},\mathscr{S})$, which concludes the proof.
\end{proof}

\begin{lemma}\label{lem:A2}
For any $z\in\mathscr{S}$, we have
$$\frac{\varphi^{4}}{N^2}
\int_0^{t}\rd s\int_{-2}^2\frac{\rd \rho(x)}{|z_{t-s}-x|^4\max(\kappa(x),s^2)}
\leq
C\frac{\kappa(E)}{\max(\kappa(E),t^2)}.
$$
\end{lemma}

\begin{proof}
We denote $z=E+\ii y$ and abbreviate $\eta(w)=\im(w)$.

Assume first that $t>2\kappa(E)^{1/2}$. 
We decompose the above integral into
\begin{equation}\label{eqn:split}
\int_0^{\kappa(E)^{1/2}}\rd s\int_{-2}^2\frac{\rd \rho(x)}{|z_{t-s}-x|^4\kappa(x)}
+
\int_{{\kappa(E)^{1/2}}}^{t-\kappa(E)^{1/2}}\rd s\int_{-2}^2\frac{\rd \rho(x)}{|z_{t-s}-x|^4 s^2}
+
\int_{{t-\kappa(E)^{1/2}}}^t\rd s\int_{-2}^2\frac{\rd \rho(x)}{|z_{t-s}-x|^4 s^2}.
\end{equation}
To evaluate the above terms, we can restrict our attention to $w$ such that assume $\re w,\im w>0$ and note that (remember $a(w)=d(w,[-2,2])$)
\begin{align*}
&\int \frac{\rd \rho(x)}{|w-x|^4}\leq \frac{1}{a(w)^2\eta(w)}\im m_{\rm sc}(w)\sim\frac{1}{a(w)^2\eta(w)}\left(\kappa(w)^{1/2}\mathds{1}_{\re w<2}+\frac{\eta(w)}{\kappa(w)^{1/2}}\mathds{1}_{\re w>2}\right),\\
&\int_{-2}^2\frac{\rd \rho(x)}{|w-x|^4\kappa(x)}\sim\frac{1}{a(w)^2\eta(w)}\im \int_{-2}^2\frac{\rd x}{(w-x)\kappa(x)^{1/2}}\leq\frac{C}{a(w)^2\eta(w)\kappa(w)^{1/2}},
\end{align*}
where in the last line we used $\int_0^\infty\frac{\rd x}{(w-x)\sqrt{x}}=\pi w/(-w)^{3/2}$.

The first term in (\ref{eqn:split}) is of order at most (we use Lemma \ref{lem:zt} to estimate $a(z_t)$, $\eta(z_t)$ and $\kappa(z_t)$)
\begin{align*}
&\kappa(E)^{1/2}\int_{-2}^2\frac{\rd \rho(x)}{|z_t-x|^4\kappa(x)}\leq 
\kappa(E)^{1/2}\frac{1}{a(z_t)^2\eta(z_t)\kappa(z_t)^{1/2}}
\leq
\frac{\kappa(E)}{yt^6}.
\end{align*}
The is negligible if $\varphi^4\kappa/(yN^2 t^6)<C \kappa/t^2$, which is true for $t>2 \kappa^{1/2}$ and $\kappa>\varphi^2 N^{-2/3}$.

The second term in (\ref{eqn:split}) is bounded by
$$
\int_{{\kappa(E)^{1/2}}}^{t-\kappa(E)^{1/2}}\frac{\rd s}{\kappa(z_{t-s})^{5/2}s^2}
\leq \int_{{\kappa(E)^{1/2}}}^{t-\kappa(E)^{1/2}}\frac{\rd s}{(t-s)^{5}s^2}\leq \frac{C}{t^2\kappa^2},
$$
negligible provided $\varphi/(N t\kappa)\leq C\kappa^{1/2}/t$, true as $\kappa>\varphi^2 N^{-2/3}$.

Finally, the last term is at most
$$
\int_{t-\kappa(E)^{1/2}}^t\frac{\kappa(z_{t-s})^{1/2}}{\eta(z_{t-s})^3s^2}\rd s 
\leq \int_{t-\kappa(E)^{1/2}}^t\frac{\kappa(E)^{1/2}}{(y+\kappa(E)^{1/2}(t-s))^{3}s^2}\leq \frac{1}{t^2}\int_0^{\kappa(E)^{1/2}}\frac{\kappa(E)^{1/2}\rd s}{(y+\kappa(E)^{1/2}s)^{3}}\leq \frac{1}{t^2 y^2},
$$
which clearly is negligible provided $\frac{\varphi}{Nyt}\leq C\kappa^{1/2}/t$, which holds as $y=\varphi^2/(N\kappa^{1/2})$. This concludes the case $t>2\kappa(E)^{1/2}$.

If $t<2\kappa(E)^{1/2}$, our integral is bounded by
$$
\int_{0}^t\rd s\int_{-2}^2\frac{\rd \rho(x)}{|z_{t-s}-x|^4 \kappa(x)}\leq 
\int_{0}^t\frac{C}{\eta(z_{t-s})^3\kappa(z_{t-s})^{1/2}}\rd s 
\leq
\int_{0}^t\frac{C}{(y+s\kappa(E)^{1/2})^3\kappa(E)^{1/2}}\rd s 
\leq \frac{C}{y^2\kappa(E)}.
$$
This term is negligible because $y=\varphi^2/(N\kappa(E)^{1/2})$.
\end{proof}

\end{document}